\documentclass[11pt,a4paper]{article}
\usepackage{graphicx}

\usepackage{fullpage,amsmath,amsfonts,amssymb,amsthm}
\usepackage{arydshln,cite}

\usepackage{cite}
\usepackage{braket} 

\numberwithin{equation}{section}

\newtheorem{Proposition}{Proposition}[section]
\newtheorem{Theorem}[Proposition]{Theorem}  
\newtheorem{Lemma}[Proposition]{Lemma}  
\newtheorem{Corollary}[Proposition]{Corollary}  
\newtheorem{Definition}[Proposition]{Definition}  

\newcommand{\nc}{\newcommand}
\nc{\chit}{\raisebox{0.25ex}{$\chi$}}
\nc{\chih}{\raisebox{0.25ex}{$\hat\chi$}}
\nc{\Bf}{\mathfrak{B}}
\nc{\Cf}{\mathfrak{C}}
\nc{\g}{\mathfrak{g}}
\nc{\G}{\mathfrak{G}}
\nc{\h}{\mathfrak{h}}
\nc{\Hf}{\mathfrak{H}}
\nc{\kf}{\mathfrak{k}}
\nc{\m}{\mathfrak{m}}
\nc{\n}{\mathfrak{n}}
\nc{\rf}{\mathfrak{r}}
\nc{\s}{\mathfrak{s}}
\nc{\Sf}{\mathfrak{S}}
\nc{\vf}{\mathfrak{v}}
\nc{\zf}{\mathfrak{z}}
\nc{\Zf}{\mathfrak{Z}}
\nc{\C}{\mathbb{C}}
\nc{\Nb}{\mathbb{N}}
\nc{\Z}{\mathbb{Z}}
\nc{\Zkl}{\mathbb{Z}_{k,l}}
\nc{\Lc}{\mathcal{L}}
\nc{\Oc}{\mathcal{O}}
\nc{\Pc}{\mathcal{P}}
\nc{\Wc}{\mathcal{W}}
\nc{\Zc}{\mathcal{Z}}
\nc{\Obar}{\overline{\Oc}}
\nc{\z}{\{0\}}
\nc{\Vp}{V_\bullet}
\nc{\Wp}{W_\bullet}
\nc{\ad}{\mathrm{ad}}
\nc{\spa}{\mathrm{span}}
\nc{\hr}{\mathrm{h}}
\nc{\het}{\mathrm{ht}}
\nc{\sr}{\mathrm{s}}

\begin{document}

\thispagestyle{empty}

\begin{center}

\scalebox{1.05}{\textbf{\huge Finite-dimensional $\mathbb{Z}$-graded Lie algebras}}
\\[0.6cm]
\scalebox{1}{\Large Mark D. Gould${}^a$, Phillip S. Isaac${}^a$, Ian Marquette${}^b$, J{\o}rgen Rasmussen${}^a$}
\\[0.4cm]
${}^a$\,\textit{School of Mathematics and Physics, University of Queensland\\ 
St Lucia, Brisbane, Queensland 4072, Australia}
\\[0.3cm] 
${}^b$\,\textit{Department of Mathematical and Physical Sciences, La Trobe University\\ 
Melbourne, Victoria 3086, Australia}
\\[0.3cm]
{\small
\textsf{m.gould1\!\;@\!\;uq.edu.au\quad\;
psi\!\;@\!\;maths.uq.edu.au\quad\;
i.marquette\!\;@\!\;latrobe.edu.au\quad\;
j.rasmussen\!\;@\!\;uq.edu.au}}

\vspace{1.4cm}

{\large\textbf{Abstract}}\end{center}

\noindent
We investigate the structure and representation theory of finite-dimensional $\mathbb{Z}$-graded Lie algebras,
including the corresponding root systems and Verma, irreducible, and Harish-Chandra modules.
This extends the familiar theory for finite-dimensional semisimple Lie algebras to a much wider class of Lie algebras,
and opens up for advances and applications in areas relying on ad-hoc approaches.
Physically relevant examples are afforded by the Heisenberg and conformal Galilei algebras, 
including the Schr{\"o}dinger algebras, whose $\mathbb{Z}$-graded structures are yet to be fully exploited.

\newpage

\tableofcontents

\section{Introduction}

We investigate the structure and representation theory of finite-dimensional $\Z$-graded Lie algebras over the complex field.
Our work and approach is partly motivated by our study of the physically important Schr\"odinger and 
conformal Galilei algebras, and continues work of Kac \cite{Kac68}; see also \cite{EK13,CC18}.
However, we emphasise that some of our definitions differ from those used previously and that, 
unlike in the work of Kac, our algebras are assumed finite-dimensional but, crucially, not necessarily semisimple. 

Here, we demonstrate that much of the theory developed for finite-dimensional semisimple Lie algebras 
(see, e.g., \cite{Serre87,Hum94}) is applicable, 
in a suitably generalised form, to a much wider class of Lie algebras, including some of great physical interest.
Indeed, there are many examples of $\Z$-graded Lie algebras naturally arising in applications. 
This includes the finite-dimensional simple Lie algebras which are naturally $\Z$-graded with respect to 
a given simple root (or subset of simple roots), for example.
In fact, this extends to a large number of non-semisimple Lie algebras such as 
the Heisenberg algebra, the Schr\"odinger and conformal Galilei algebras, as well as the model filiform 
Lie algebras. While there has been a great deal of activity in the study of these algebras due to 
their importance in physics, their $\Z$-graded structures have not been fully exploited, and a systematic study 
has so far not appeared in the literature. Although generally not reductive, these algebras
nevertheless have many properties in common with ordinary semisimple Lie algebras, as we will discuss.

Elucidating the power of the structure theory presented here, 
we explore the representation theory of the class of so-called  \textit{normal} $\Z$-graded Lie algebras,
and demonstrate that a suitable character theory can be developed in this setting. 
In doing so, we provide a cornerstone for future studies that aligns with our original motivation for the current work, 
acknowledging that it is often the representations that manifest in applications of physical interest.

The paper is set up as follows. Section~\ref{Sec:Prelim} is concerned with the interplay between the $\Z$-graded 
structure and the Levi decomposition of a Lie algebra.
This problem has been previously considered in \cite{CC18}, but our work offers the novelty of a description 
in terms of the radical of the Killing form rather than the maximal solvable radical. 
This distinction is important for the remainder of the paper. Motivated by applications to physics, 
Section~\ref{Sec:Normal} is devoted to the introduction and structure of \textit{normal} $\Z$-graded Lie algebras, 
while Sections~\ref{Sec:Regular} and~\ref{Sec:Connected} are devoted to the introduction and structure theory of \textit{regular} 
and \textit{irreducible} $\Z$-graded Lie algebras, respectively.
Section~\ref{Sec:Ex} follows up with some important examples. 
The representation theory of normal $\Z$-graded Lie algebras is developed in 
Section~\ref{Sec:Rep}, with emphasis on a new category of $\Z$-graded modules, which we call Category $\Zc$. 
This includes Category $\Oc$, familiar from the theory of simple Lie algebras but extended to the current setting.
Section~\ref{Sec:Chacters} concerns the characters of several key classes of modules, including Harish-Chandra modules.
The paper concludes in Section~\ref{Sec:Concl} with a brief outlook to future developments.

\section{Preliminaries}
\label{Sec:Prelim}

\subsection{Notation}
\label{Sec:Notation}

Throughout, we let $L$ denote a finite-dimensional complex Lie algebra with radical $R$, 
nilradical $\n$, centre $Z(L)$, derived algebra $L'$ and Killing form $(\ ,\,)$.
We note that $R$ is the orthocomplement of $L'$ under the Killing form, in the sense that
\begin{align}
 R=\{x\in L\,|\,(x,L')=\z\}.
\end{align}
For any subspace $A\subseteq L$, the normaliser is defined by
\begin{align}
 N_L(A):=\{x\in L\,|\,[x,A]\subseteq A\},
\end{align}
while for any subset $B\subseteq L$, the centraliser is defined by
\begin{align}
 C_L(B):=\{x\in L\,|\,[x,B]=\z\}.
\end{align}
In the Levi decomposition
\begin{align}
 L=S\oplus R,
\end{align}
$S$ is a semisimple subalgebra of $L$, called the Levi subalgebra.
This decomposition is also an $S$-module decomposition.
An element of $L$ is said to be semisimple on a subset $B\subseteq L$ if it is ad-diagonalisable on $B$.
We note that if $L$ is reductive, then $L=L'\oplus Z(L)$.
We let $\h_n$ denote the $n$-dimensional Heisenberg algebra
and $\langle x_1,\ldots,x_n\rangle$ the Lie algebra generated by $\{x_1,\ldots,x_n\}\subset L$.
The set of nonnegative integers is denoted by $\Nb_0$.

\subsection{Killing decomposition}

We are concerned with the kernel of the Killing form, herein called the {\em Killing radical}.
\begin{Definition}
The Killing radical of $L$ is defined as
\begin{align}
 K:=\{x\in L\,\,|\,(x,L)=\z\}.
\end{align}
\end{Definition}

\noindent
\textbf{Remark.}
By construction, $K\subseteq R$, but equality does not generally hold.
It follows from the invariance of the Killing form that $[L,R]\subseteq K$,
so $[K,K]\subseteq[R,K]\subseteq[R,R]\subseteq K$.
In fact,
\begin{align}
 [L,R]=L'\cap R\subseteq\n\subseteq K.
\label{LRnK}
\end{align}
\begin{Proposition}
$R$ admits an $S$-module decomposition of the form
\begin{align}
 R=Z_s\oplus K.
\label{RZK}
\end{align}
\end{Proposition}
\begin{proof}
Since $S$ is semisimple and $K$ is an ideal of $L$, hence an $S$-submodule of $L$, the result follows.
\end{proof}

\noindent
\textbf{Remark.}
We refer to the ensuing $S$-module decomposition 
\begin{align}
 L=L_s\oplus K,\qquad L_s:=S\oplus Z_s,
\label{LLK}
\end{align}
as the {\em Killing decomposition} of $L$.
Since $S$ is semisimple, the Killing form restricted to $S$ is non-degenerate, 
so we may assume the decomposition of $L_s$ is orthogonal under the Killing form. 
We stress that $Z_s$ need not be zero.
\begin{Proposition}
\label{Prop:Ls}
The Killing form restricted to $L_s$ is non-degenerate.
\end{Proposition}
\begin{proof}
Use the Killing decomposition to write $z\in L$ as $z=z'+z''$, where $z'\in L_s$ and $z''\in K$.
If there exists nonzero $x\in L_s$ such that $(x,y)=0$ for all $y\in L_s$, then $(x,z)=(x,z')+(x,z'')=(x,z')=0$, 
so $x\in K$, a contradiction.
\end{proof}
\begin{Proposition}
\label{Prop:Killing}
\begin{align}
 [Z_s,Z_s] \subseteq K,\qquad [S,Z_s] = \z.
\end{align}
\end{Proposition}
\begin{proof}
By construction, $[S,Z_s] \subseteq Z_s$. 
The Killing form induces a non-degenerate invariant form on the quotient algebra 
\begin{align}
\overline{L}:=L/K.
\end{align}
Thus, $\overline{L}$ is reductive with centre $(Z_s \oplus K)/K$ and $\overline{L}'$ is semisimple.
It follows that $[Z_s,Z_s] \subseteq K$ and $[S,Z_s] \subseteq K$, so $[S,Z_s] \subseteq K \cap Z_s = \z$.
\end{proof}

\noindent
\textbf{Remark.}
We emphasise that $L_s$ is not generally a Lie algebra. Indeed, we have the following result.
\begin{Corollary} 
\label{Cor:LsLie}
$L_s$ is a Lie algebra if and only if $[Z_s,Z_s] = \z$. In that case, $L_s$ is reductive with 
$Z(L_s)=Z_s$ and $L_s'=S$.
\end{Corollary}
\begin{proof}
From $[S,Z_s] = \z$ and $[S,S] = S$, it follows that $L_s = S \oplus Z_s$ is a Lie algebra if and only if 
$[Z_s,Z_s] \subset L_s$. By Proposition~\ref{Prop:Killing}, this is the case if and only if 
$[Z_s,Z_s] \subseteq K \cap L_s = \z$. 
Using Proposition~\ref{Prop:Ls}, this implies $L_s$ is reductive with centre $Z_s$ and $L_s'= S$.
\end{proof}
\begin{Lemma}
\begin{align}
 Z(L) \subseteq K,\qquad [S,K] \subseteq \n,\qquad
 [L,Z_s] \subseteq K,\qquad [L_s,K] \subseteq K.
\end{align}
\end{Lemma}
\begin{proof}
The first inclusion follows from the fact that $(z,L)=\z$ for all $z\in Z(L)$.
The second inclusion follows from $[S,K]\subseteq[L,R]$ and \eqref{LRnK}.
The third inclusion follows from \eqref{LRnK} and Proposition~\ref{Prop:Killing}, 
while the last inclusion follows from $[L_s,K]\subseteq[L,R]$ and \eqref{LRnK}.
\end{proof}
\begin{Proposition}
\label{Prop:ZsR}
\mbox{}
\begin{itemize}
\item[{\rm (i)}]
$Z_s \subseteq C_L(S) \subseteq R$.
\item[{\rm (ii)}]
$C_L(S) = Z_s \oplus C_L(S) \cap K$.
\item[{\rm (iii)}]
$R = Z_s \oplus C_L(S) \cap K \oplus [S,K]$.
\end{itemize}
\end{Proposition}
\begin{proof}
Clearly, $Z_s \subseteq C_L(S)$. To show $C_L(S) \subseteq R$, let
$c\in C_L(S)$ and use the Levi decomposition to write $c=x+r$, where $x \in S$ and $r \in R$. Then,
$\z = [c,S] = [x,S] + [r,S]$, so
\begin{align}
 [x,S] = [r,S] \subseteq S \cap R = \z.
\end{align}
Since $S$ is semisimple, this implies $x=0$ and $c = r \in R$. This establishes (i). 
As to (ii), let $c \in C_L(S)$ and use $C_L(S)\subseteq R=Z_s\oplus K$ to write $c = z + y$, where $z \in Z_s$ 
and $y\in K$. Then, $\z = [c-z,S] = [y,S]$, so $y\in C_L(S)$, hence
\begin{align}
 c \in Z_s \oplus C_L(S) \cap K
\end{align} 
and thus $C_L(S) \subseteq Z_s \oplus C_L(S) \cap K$. The reverse inclusion is obvious. 
As to (iii), since $S$ is semisimple and $K$ is stable under the adjoint action of $S$, we may write 
\begin{align}
K = C_L(S) \cap K \oplus [S,K],
\end{align}
readily implying (iii).
\end{proof}

\noindent
\textbf{Remark.} 
Proposition~\ref{Prop:ZsR} helps to clarify the connection between $R$ and $K$. 
Similar results hold for
the kernel of any nonzero invariant bilinear form on $L$, not just the Killing form.

\subsection{Gradation}

We now introduce the main topic of the paper.
\begin{Definition}
\label{Def:Zgraded}
$L$ is called $\Z$-gradable if it admits a $\Z$-grading of the form
\begin{align}
 L=\bigoplus_{i=-k}^{l}L_i, \qquad [L_i,L_j] \subseteq L_{i+j},\qquad k,l\in\Nb_0,
\label{Zgrad}
\end{align}
where $L_{-k}\neq\z$, $L_l\neq\z$, and $L_i\equiv\z$ for $i<-k$ or $i>l$. 
Such a gradation is called \textit{balanced} if $k = l$.
$L$ is said to be $\Z$-graded if it comes equipped with a $\Z$-grading.
\end{Definition}
\noindent
\textbf{Remark.}
Associated with a $\Z$-grading of $L$ is a derivation $D$ satisfying $D(x) = ix$ for all $x\in L_i$ for each $i$. 
If the derivation is not inner (i.e., not of the form $D=\ad_d$ for any $d\in L$), 
then we extend the Lie algebra by enlarging its basis from $\Bf$ to $\Bf\cup\{d\}$, where
\begin{align}
 [d,x]=D(x),\qquad \forall x \in L.
\end{align}
In the following, we shall accordingly assume that it \textit{is} inner.
It follows that this \textit{level operator} $d$ is semisimple on $L$ and that $d \in Z(L_0)$.  
\medskip

\noindent
\textbf{Remark.}
We view two $\Z$-gradings of a given Lie algebra $L$ as equivalent if one of the corresponding level operators (say $d_1$)
is a positive integer multiple of the other (say $d_2$).
Although multiplying a level operator by $-1$ also gives rise to a $\Z$-grading, 
we will in general view the ensuing \textit{reverse $\Z$-grading} as distinct from the original $\Z$-grading.
\medskip

\noindent
\textbf{Remark.} 
To specify $k$ and $l$ in \eqref{Zgrad}, we occasionally refer to $L$ as $\Zkl$-graded.
For convenience, we introduce
\begin{align}
 m:=\min\{k,l\}.
\end{align}
\noindent
\textbf{Remark.}
Given a $\Z$-grading, $L_0$ is a Lie subalgebra of $L$ and each $L_i$ is an $L_0$-module.
While $\dim L>0$, we allow the possibility that $L_i = \z$ for some $i$ such that $-k<i<l$.
In the remainder of this paper, $L$ is assumed $\Zkl$-graded for some arbitrary but fixed $k,l\in\Nb_0$.
\begin{Proposition}\label{prop:LL0}
The decomposition \eqref{Zgrad} is orthogonal with respect to the Killing form, in the sense that
\begin{align}
(L_i,L_j) = \z,\qquad i+j \neq 0.
\end{align}
\end{Proposition}
\begin{proof}
This follows readily from the invariance of the Killing form.
\end{proof}
\begin{Definition}
For each $i$, we define
\begin{align}
\g_i:=L_s \cap L_i,\qquad
\kf_i:=K \cap L_i,\qquad
\rf_i:=R \cap L_i,\qquad 
\s_i:=S \cap L_i.
\label{ISS}
\end{align}
\end{Definition}
\noindent
\textbf{Remark.}
The Killing decomposition of $L$ induces the graded decompositions
\begin{align}
 L_s=\bigoplus_i \g_i,\qquad 
 K=\bigoplus_i \kf_i,\qquad 
 R=\bigoplus_i \rf_i,\qquad 
 S=\bigoplus_i \s_i,
\label{LK}
\end{align}
while
\begin{align}
 L_i=\g_i\oplus\kf_i,\qquad L_i = \s_i \oplus \rf_i,\qquad \forall i.
\label{Li}
\end{align}
\begin{Proposition}
\label{Prop:gskr}
\mbox{}
\begin{itemize}
\item[{\rm (i)}]
For each $i \neq 0$, $\g_i=\s_i$ and $\kf_i = \rf_i$.
\item[{\rm (ii)}]
$\s_0$ is a reductive Lie algebra.
\item[{\rm (iii)}]
We have $\s_0$-module decompositions
\begin{align}
 L_0=\g_0 \oplus \kf_0,\qquad \g_0 = \s_0 \oplus Z_s,\qquad \rf_0 = Z_s \oplus \kf_0,
\end{align}
where
\begin{align}
  [\s_0,Z_s] = \z,\qquad [Z_s,Z_s] \subseteq \kf_0.
\label{Zsz}
\end{align}
\end{itemize} 
\end{Proposition}
\begin{proof}
For $i \neq 0$, observe that $\kf_i = K \cap L_i \subseteq R \cap L_i=\rf_i $. 
But for $x \in L_i$, $[d,x] = ix$, so $\rf_i = [d,\rf_i]$. 
Thus, $(L,\rf_i) = (L,[d,\rf_i]) = ([L,d],\rf_i)=\z$, 
since $\rf_i\subseteq R$ is orthogonal to $[d,L] \subseteq L'$. 
It follows that $\rf_i \subseteq K$, so $\rf_i \subseteq K \cap L_i = \kf_i$, 
hence $\kf_i=\rf_i$. 
Together with \eqref{Li}, this implies (i).
Part (ii) follows from the observation that $\s_0$ is a Lie algebra and that the Killing form 
of $S$ restricted to $\s_0$ is non-degenerate.
As to this non-degeneracy, since $d\in L_0=\s_0\oplus\rf_0$, we may write $d=d_0+r_0$, 
where $d_0\in\s_0$ and $r_0\in\rf_0$, so for any $x_i\in\s_i$, we have
\begin{align}
 \underbrace{ix_i-[d_0,x_i]}_{\in S}=\underbrace{[r_0,x_i]}_{\in R},
\end{align}
and since $S\cap R=\z$, we see that $d_0$ is a level operator for $S$: $[d_0,x_i]=ix_i$.
It follows that $(\s_0,\s_i)=\z$ for all $i\neq0$, and since $S$ is semisimple, its Killing form is non-degenerate,
so its restriction to $\s_0$ is as well.
As to (iii), we first show that $Z_s \subseteq \g_0$. Suppose $z \in Z_s$ and use \eqref{Zgrad} to write
\begin{align}
 z = z_0 + \sum_{i \neq 0}z_i,
\end{align}
where $z_i \in Z_s \cap L_i \subseteq \rf_i$. 
From (ii), for $i \neq 0$, we have $\rf_i = \kf_i$, 
so $z_i \in \g_i \cap \kf_i \subseteq L_s \cap K = \z$, hence $z = z_0 \in L_0$,
implying that $Z_s \subseteq L_0 \cap R = \rf_0$. 
Note that $L_0 = \s_0 \oplus \rf_0$ and from \eqref{LLK}, $L_0=\g_0 \oplus \kf_0$. 
Moreover, by \eqref{RZK} and \eqref{LLK}, we have the $\s_0$-module decompositions
\begin{align}
 \g_0 = \s_0 \oplus Z_s,\qquad 
 \rf_0 = Z_s \oplus \kf_0.
\end{align} 
Finally, by Proposition~\ref{Prop:Killing}, $[Z_s,Z_s] \subseteq \rf_0 \cap K = \kf_0$ 
and $[\s_0,Z_s] \subseteq [S,Z_s] = \z$.
\end{proof}

\noindent
\begin{Corollary}
\label{Cor:g0}
The following conditions are equivalent:
\begin{itemize}
\item[{\rm (i)}] $\g_0$ is a Lie algebra.
\item[{\rm (ii)}] $[Z_s,Z_s] = \z$.
\item[{\rm (iii)}] $L_s$ is a Lie algebra.
\end{itemize}
\end{Corollary}
\begin{proof}
As to (i)\,$\Rightarrow$\,(ii), for $\g_0$ a Lie algebra, we have $[Z_s,Z_s] \subseteq \g_0$, 
hence $[Z_s,Z_s] \subseteq \g_0 \cap \kf_0 = \z$.
As to (ii)\,$\Rightarrow$\,(i), it follows from $\g_0 = \s_0 \oplus Z_s$ that $[Z_s,Z_s] = \z$ implies 
that $\g_0$ is a Lie algebra. 
The equivalence of (ii) and (iii) is proved in Corollary~\ref{Cor:LsLie}.
\end{proof}
\begin{Proposition}
\label{Prop:gonondeg}
The Killing form restricted to $\g_0$ is non-degenerate.
\end{Proposition}
\begin{proof}
We have $L_0=\g_0 \oplus \kf_0$, and the kernel of the Killing form restricted to $L_0$ is $\kf_0$, 
so the Killing form restricted to $\g_0$ is non-degenerate.
\end{proof}
\noindent
\textbf{Remark.}
By Proposition~\ref{Prop:gonondeg}, if $\g_0$ is a Lie algebra, then it must be reductive.
\begin{Proposition}
\label{Prop:gkk}
\mbox{}
\begin{itemize}
\item[{\rm (i)}]
$[L_i,\kf_j] \subseteq \kf_{i+j}$ for all $i,j$.
\item[{\rm (ii)}]
$[\g_i,\g_j] \subseteq \g_{i+j}$ for all $i,j \neq 0$.
\item[{\rm (iii)}]
$[\s_0,\g_i] \subseteq \g_i$ for all $i\neq0$.
\item[{\rm (iv)}]
$[Z_s,\g_i] \subseteq \kf_i$ for all $i$.
\end{itemize}
\end{Proposition}
\begin{proof}
Using that $[L_i,L_j]\subseteq L_{i+j}$ for all $i,j$, part (i) follows from $[L,K]\subseteq K$;
part (ii) from $\g_i=\s_i$ for all $i\neq0$, $\s_0\subseteq\g_0$ and that $S$ is a Lie subalgebra;
part (iii) from $\g_i=\s_i$ for all $i\neq0$ and that $S$ is a Lie subalgebra; and
part (iv) from $\g_i=\s_i$ for all $i\neq0$, so $[Z_s,\g_i] = [Z_s,\s_i] \subseteq [Z_s,S] = \z$ for $i \neq 0$ 
and for $i = 0$, $[Z_s,\g_0] =[Z_s,Z_s] \subseteq \kf_0$.
\end{proof}
\begin{Corollary}
\label{Cor:Zs'}
\begin{align}
 [Z_s,Z_s]= \z\quad\Longleftrightarrow\quad [\g_0,\g_i] \subseteq \g_i,\quad \forall i.
\end{align}
\end{Corollary}
\begin{proof}
By Corollary~\ref{Cor:LsLie}, $[Z_s,Z_s] = \z$ if and only if $L_s$ is a Lie algebra. 
If $L_s$ is a Lie algebra, then \eqref{LK} implies that $[\g_0,\g_i] \subseteq L_s \cap L_{i} = \g_i$. 
Conversely, if $[\g_0,\g_i] \subseteq \g_i$ for 
all $i$, then Proposition~\ref{Prop:gkk} implies that $L_s$ is closed under the Lie bracket.
\end{proof}

\section{Normal algebras}
\label{Sec:Normal}

For simplicity, we will write $C_0(\g_0)\equiv C_{L_0}(\g_0)$ and 
$N_0(\g_0)\equiv N_{L_0}(\g_0)$ for the centraliser and normaliser of $\g_0$ in $L_0$, 
as well as $C(L_s)\equiv C_L(L_s)$ and $N(L_s)\equiv N_L(L_s)$ for the
centraliser and normaliser of $L_s$ in $L$. We also note that $N_0(\g_0)= N_L(\g_0) \cap L_0$. An important algebra in the following is
$$
\Cf := C_0(\g_0) \cap K = C_0(\g_0) \cap \kf_0.
$$
\begin{Definition}
\label{Def:normal}
A $\Zkl$-graded Lie algebra $L$ is called normal if it has the following properties:
$$
\begin{array}{rll}
{\rm (i)}\!\!\!& \text{reductivity:} & 
\mbox{$\g_0$ is a reductive Lie algebra.}
\\[.2cm]
{\rm (ii)}\!\!\!& \text{complete reducibility:} & 
\mbox{$L$ is completely reducible as a $\g_0$-module}.
\\[.2cm]
{\rm (iii)}\!\!\!& \text{multiplicity free:} & 
\mbox{For $i \neq 0$, $L_i$ admits a multiplicity free decomposition into}
\\[.05cm] && \mbox{irreducible $\g_0$-modules.} 
\\[.2cm]
{\rm (iv)}\!\!\!& \text{$\Cf$-central:} &
\mbox{$Z(\Cf) \subseteq Z(L)$.} 
\\[.2cm]
{\rm (v)}\!\!\!& \text{non-singularity:} &
\mbox{For each index $i$ such that $0 < |i| \leq m$, $[\kf_i,\kf_{-i}] \subseteq N(L_s)$.}
\\[.2cm]
\end{array}
$$
\end{Definition}
\noindent
\textbf{Remark.}
As $Z(L) \subseteq Z(\Cf)$, condition (iv) implies $Z(\Cf) = Z(L)$.
\begin{Theorem}
\label{Theo:d}
If $L$ is a normal $\Z$-graded Lie algebra, then $d \in Z(\g_0) \oplus Z(L)$.
\end{Theorem}
\begin{proof}
As $d\in L_0=\g_0 \oplus \kf_0$, we may write $d = d_0 + c$ with $d_0 \in \g_0$ and $c \in \kf_0$. 
From $[d,\g_0]=\z$, it follows that $[d_0,\g_0] = [c,\g_0] \subseteq \g_0 \cap \kf_0 = \z$, 
so $d_0 \in Z(\g_0)$ and $c \in C_0(\g_0) \cap K = \Cf$. 
Moreover, $\z = [d,\Cf] = [d_0,\Cf] + [c,\Cf] = [c,\Cf]$, so $c \in Z(\Cf) = Z(L)$. The result now follows.
\end{proof}
\noindent
In the remainder of this section, $L$ will be assumed normal.
\begin{Theorem}
\label{Theo:Ls}
$L_s$ is a reductive Lie algebra with $Z(L_s)=Z_s$.
\end{Theorem}
\begin{proof}
Since $\g_0$ is a Lie algebra, Corollary~\ref{Cor:g0} implies $[S,Z_s] = [Z_s,Z_s] = \z$, so 
$[L_s,Z_s] = \z$, hence $Z_s \subseteq Z(L_s)$. Since $L_s = S \oplus Z_s$ and $S$ is semisimple, 
the result follows.
\end{proof}
\begin{Corollary}
\label{Cor:Z_s}
$Z(\g_0) = Z(\s_0) \oplus Z_s$.
\end{Corollary}
\begin{proof}
We have $\g_0 = \s_0 \oplus Z_s$ where $Z_s = Z(L_s) \subseteq Z(\g_0)$. 
By Proposition~\ref{Prop:gskr}, $\s_0$ is a reductive Lie algebra, so $\s_0 = {\s_0}' \oplus Z(\s_0)$, 
hence $\g_0 = {\s_0}' \oplus Z(\s_0) \oplus Z_s$. 
Since ${\s_0}'$ is semisimple, the result follows.
\end{proof}

\noindent
The $\Z$-gradings \eqref{LK} now imply the following strengthened version of Proposition~\ref{Prop:gkk}.
\begin{Proposition}
\begin{align}
 [\g_i,\g_j] \subseteq \g_{i+j},\qquad [\g_i,\kf_j] \subseteq \kf_{i+j},\qquad [\kf_i,\kf_j] \subseteq \kf_{i+j}.
\end{align}
\end{Proposition}

\noindent
As to the algebra $\Cf$, we have the following result.
\begin{Proposition}
\label{Prop:C0}
\mbox{}
\begin{itemize}
\item[{\rm (i)}]
$C_0(\g_0) = Z(\g_0) \oplus \Cf$.
\item[{\rm (ii)}]
$N_0(\g_0) = \g_0 \oplus \Cf$.
\end{itemize}
\end{Proposition}
\begin{proof}
Let $c \in C_0(\g_0)$. As $C_0(\g_0)\subseteq L_0=\g_0 \oplus \kf_0$, we may write $c = x + y$ with
$x \in \g_0$ and $y\in \kf_0$. From $[c,\g_0]=\z$, it follows that
$[x,\g_0] =[y,\g_0] \subseteq \g_0 \cap K = \z$.
Since $\g_0$ is reductive, we thus have $x \in Z(\g_0)$ and 
$y\in C_0(\g_0) \cap K =\Cf$, so $C_0(\g_0) \subseteq Z(\g_0) \oplus \Cf$. 
As the reverse inclusion is obvious, (i) follows. As to (ii), suppose $n \in N_0(\g_0) \cap K$ and write 
$n = x + y$, where $x \in \g_0$ and $y\in \kf_0$. Then, $[n-x,\g_0] = [y,\g_0] \subseteq  \g_0 \cap K = \z$. 
Thus, $y\in C_0(\g_0) \cap K = \Cf$, so $n \in \g_0 \oplus \Cf$, 
hence $N_0(\g_0) \subseteq \g_0 \oplus \Cf$. The reverse inclusion is clear. 
\end{proof}
\begin{Corollary}
$\Cf=N_0(\g_0) \cap K$.
\end{Corollary}
\begin{proof}
Let $n \in N_0(\g_0) \cap K$, and use Proposition \ref{Prop:C0} to write $n = x + c$, where $x\in\g_0$ and $c \in\Cf$.
It follows that $x\in K \cap \g_0 = \z$, so $n = c \in\Cf$, hence $N_0(\g_0) \cap K \subseteq\Cf$. 
Conversely, $C_0(\g_0) \subseteq N_0(\g_0)$, so $C_0(\g_0) \cap K \subseteq N_0(\g_0) \cap K$,
and the result follows.
\end{proof}
\begin{Proposition}
\label{Prop:CN}
\mbox{}
\begin{itemize}
\item[{\rm (i)}]
$C(L_s) = Z_s \oplus (C(L_s) \cap K)$.
\item[{\rm (ii)}] 
$N(L_s) = L_s \oplus (C(L_s) \cap K)$.
\end{itemize}
\end{Proposition}
\begin{proof}
Let $c \in C(L_s)$ and use the Killing decomposition to write $c = x +y$, where $x \in L_s$ and 
$y\in K$. Then, $[c,L_s] = \z = [x,L_s] + [y,L_s]$, so $[x,L_s] =[y,L_s] \subseteq L_s \cap K = \z$.
Thus, $x \in Z_s$ and $y\in C(L_s) \cap K$, so $C(L_s) \subseteq Z_s \oplus C(L_s) \cap K$. 
As the reverse inclusion is clear, (i) follows. Part (ii) follows similarly.
\end{proof}
\begin{Proposition}
\label{Prop:Cf}
$\Cf= C(L_s) \cap K = N(L_s) \cap K$.
\end{Proposition}
\begin{proof}
Let $c \in\Cf$. For $i \neq 0$, let $M_i \subseteq \g_i$ be an irreducible $\g_0$-module. 
Then, $\theta$: $M_i \rightarrow \kf_i$, $x\mapsto [c,x]$, is a $\g_0$-module homomorphism. 
If $\theta\neq0$, then $\theta(M_i)$ is isomorphic to $M_i$ as $\g_0$-modules, in 
contradiction to condition (iii) in Definition~\ref{Def:normal}. 
Thus, $[c,M_i] = \z$ and hence $[c,\g_i] = \z$ for all $i \neq 0$, so 
$[c,\g_i] = \z$ for all $i$. This shows that $c \in C(L_s) \cap K$ and thus 
$\Cf\subseteq C(L_s) \cap K$. The reverse inclusion is clear. 
By Proposition~\ref{Prop:CN}, we also have $C(L_s) \cap K = N(L_s) \cap K$.
\end{proof} 
\begin{Corollary}
\label{Cor:NLLC}
\mbox{}
\begin{itemize}
\item[{\rm (i)}]
$C(L_s) = Z_s \oplus \Cf$.
\item[{\rm (ii)}] 
$N(L_s) = L_s \oplus \Cf$.
\end{itemize}
\end{Corollary}
\begin{proof}
The result is an immediate consequence of Proposition \ref{Prop:CN} and \ref{Prop:Cf}.
\end{proof}
\begin{Proposition}
For each index $i$ such that $0 < |i| \leq m$, $[\kf_i,\kf_{-i}] \subseteq \Cf$.
\end{Proposition}
\begin{proof}
By Proposition \ref{Prop:Cf} and condition (v) in Definition~\ref{Def:normal},
$[\kf_i,\kf_{-i}] \subseteq N(L_s) \cap K =\Cf$.
\end{proof}

\subsection{Weights and roots}

\begin{Proposition}
\label{Prop:h0}
Let $H_s$ be a Cartan subalgebra (CSA) of $\s_0'$. Then,
\begin{align}
 \h_0:= H_s \oplus Z(\g_0)
\end{align}
is a CSA of $\g_0$.
\end{Proposition}
\begin{proof}
Since $\g_0$ is reductive, it admits $H(\g_0')\oplus Z(\g_0)$ as a CSA, where $H(\g_0')$ is a CSA of $\g_0'$. 
By Proposition~\ref{Prop:gskr} (iii) and  Corollary~\ref{Cor:Z_s}, $\g_0'=\s_0'$, and the result follows.
\end{proof}
\begin{Proposition}
The Killing form restricted to $\h_0$ is non-degenerate.
\end{Proposition}
\begin{proof}
By Proposition~\ref{Prop:gonondeg}, the restriction of the Killing form to $\g_0$ is non-degenerate, 
and since $\g_0$ is reductive with CSA $\h_0$, the result follows.
\end{proof}
\begin{Definition}
We refer to
\begin{align}
 H:= \h_0 \oplus Z(L)
\end{align}
as the CSA of $L$.
\end{Definition}
\begin{Theorem}
$d \in H$.
\end{Theorem}
\begin{proof}
By Theorem \ref{Theo:d}, $d\in Z(\g_0)\oplus Z(L)\subseteq H$.
\end{proof}
\begin{Proposition}
Every element of $H$ is semisimple on $L$. 
\end{Proposition}
\begin{proof}
By condition (ii) in Definition~\ref{Def:normal}, $\h_0$ is semisimple on $L$. Since $[Z(L),L]=\z$, the result follows.
\end{proof}
\noindent
\textbf{Remark.}
Unlike for semisimple Lie algebras, the CSA $H$ need not equal its own centraliser. 
Indeed, there may exist zero-weight vectors in $K$ which are not in $Z(L)$. 
Nevertheless, as seen in Theorem~\ref{Theo:NH}, $N_0(H)=C_0(H)$.
\medskip

\noindent
\textbf{Remark.}
Let $V$ be an $L$-module. As usual, a nonzero vector $v \in V$ a called a {\em weight vector} of weight 
$\Lambda \in H^*$ if $hv = \Lambda(h) v$ for all $h \in H$.
\medskip

\noindent
\textbf{Remark.}
The Weyl group $\Wc_0$ of $\g_0$, herein referred to as the Weyl group of $L_0$, has a natural 
action on the weights $\Lambda \in H^*$. Indeed, relative to the decomposition $H^*=\h_0^*\oplus Z^*$,
we write $\Lambda=\lambda+z^*$ with $\lambda\in\h_0^*$ and $z^*\in Z^*$, and on this, $\sigma\in\Wc_0$ 
acts as $\sigma(\Lambda)=\sigma(\lambda) + z^*$. 
\begin{Definition}
The roots of $L$ are defined as the nonzero weights of the adjoint representation, 
and the set of roots is denoted by $\Phi$.
For each $i$, $\Phi_i$ denotes the set of roots associated with $L_i$. 
For each $\beta \in \Phi$, the corresponding root space is defined by
\begin{align}
L_{\beta}:= \{x \in L\,\,|\,[h,x] = \beta(h)x,\,\forall h \in H\}.
\end{align}
\end{Definition}

\noindent
\textbf{Remark.}
We note that
\begin{align}
 \Phi=\bigcup_{i=-k}^{l}\Phi_i
\end{align}
and that, for every $\beta\in\Phi$,
\begin{align}
 \beta (z) =0, \qquad\forall z \in Z(L).
\end{align}

The root system $\Phi_0$ may be partitioned as
\begin{align}
 \Phi_0 = \Phi_0^\s \cup \Phi_0^\kf,
\end{align}
where $\Phi_0^\s$ is the set of roots of $\g_0$ and  $\Phi_0^\kf$ the set of roots in $\kf_0$. 
As the set of positive roots, we take 
\begin{align}
 \Phi^+:= \Phi_0^+ \cup \Phi_1^+,
\end{align}
where
\begin{align}
 \Phi_0^+:= \Phi_0^{\s,+} \cup \Phi_0^{\kf,+}
\end{align}
is the set of positive roots in $\Phi_0$ (with respect to the partial ordering induced by the positive roots of $\g_0$) 
and
\begin{align}
 \Phi_1^+:= \Phi_1\cup\cdots\cup \Phi_l.
\end{align}
Likewise, the set of negative roots is given by 
\begin{align}
\Phi^-:= \Phi_0^{-} \cup \Phi_1^-,
\end{align}
where $\Phi_0^{-}:= \Phi_0^{\s,-} \cup \Phi_0^{\kf,-}$ is the set of negative roots of $L_0$ and
\begin{align}\label{Phi1m}
\Phi_1^-:= \Phi_{-k} \cup\cdots\cup \Phi_{-1}.
\end{align}

\noindent
\textbf{Remark.}
The roots in $\kf_0$ can be partitioned 
into positive and negative roots in a way consistent with the partial ordering on weights induced by the positive 
roots of $\g_0$. Indeed, we take a $\g_0$-dominant weight
\begin{align}
 \gamma\notin\bigcup_{\beta \in \Phi_0}\Pc_\beta,
\label{hyperplane}
\end{align}
where
\begin{align}
 \Pc_\beta:=\{\mu \in H^* \,|\,(\mu,\beta)=0\}
\end{align}
is the hyperplane orthogonal to $\beta$, and then declare
\begin{align}
 \beta>0\quad &\text{if}\ \ (\beta,\gamma)>0,\\[.2cm]
 \beta<0\quad &\text{if}\ \ (\beta,\gamma)<0.
\end{align}
\begin{Definition}
\label{Def:simple}
A simple root of $L$ is a positive root that cannot be written as the sum of two positive roots.
\end{Definition}

\noindent
\textbf{Remark.}
It follows from condition (iii) in Definition~\ref{Def:normal} that the elements of $\Cf$ are semisimple 
on each $L_i$, $i \neq 0$, and thus on the $\g_0$-module
\begin{align}\label{G}
 G:=\g_0\oplus\bigoplus_{i \neq 0}L_i.
\end{align}

\section{Regular algebras}
\label{Sec:Regular}

\begin{Definition}
\label{Def:regular}
A $\Zkl$-graded Lie algebra $L$ is called regular if it has the following properties:
$$
\begin{array}{rll}
{\rm (i)}\!\!\!& \text{reductivity:} & 
\mbox{$\g_0$ is a reductive Lie algebra.}
\\[.2cm]
{\rm (ii)}\!\!\!& \text{multiplicity free:} & 
\mbox{For $i \neq 0$, $L_i$ admits a multiplicity free decomposition into}
\\[.05cm] && \mbox{irreducible $\g_0$-modules.} 
\\[.2cm]
{\rm (iii)}\!\!\!& \text{reflexivity:} & 
\mbox{If $x \in \kf_i$, $0<|i| \leq m$, satisfies $[x,\kf_{-i}] = \z$, then $x=0$.}
\\[.2cm]
{\rm (iv)}\!\!\!& \text{non-singularity:} &
\mbox{For each index $i$ such that $0 < |i| \leq m$, $[\kf_i,\kf_{-i}] \subseteq N(L_s)$.}
\\[.2cm]
{\rm (v)}\!\!\!& \text{completeness:} &
\mbox{$L$ is generated as a Lie algebra by $G$ defined in \eqref{G}.}
\end{array}
$$
\end{Definition}

\noindent
Throughout this section, we assume all $\Zkl$-graded Lie algebras are regular. 
\medskip

\noindent
\textbf{Remark.}
The completeness and multiplicity free conditions imply that $L$ is completely reducible as a $\g_0$-module.
The reflexivity and non-singularity conditions are important for the root-space 
structure of $L$ discussed below. In particular, they imply that $L_{\pm i}$, $0 < |i| \leq m$, 
are pairwise related via duality, as seen in Section~\ref{Sec:Structure}. 
\medskip

\noindent
The completeness condition (v) implies the following result.
\begin{Proposition}
\label{Prop:k0}
$\kf_0 = \n \cap L_0$.
\end{Proposition}
\begin{proof}
Since $L$ is generated by $G$, we may write
\begin{align}
L = G + [G,G] + [G,[G,G]] +\cdots = G + [G,L] = G+L'.
\end{align}
Moreover, the Levi decomposition of $L$ implies $L'\subseteq S \oplus [L,R]$, and we recall \eqref{LRnK}
and note that $S\subseteq G$.
It follows that $L = G + [L,R]$, so
\begin{align}
L_0 = G_0 + [L,R] \cap L_0 = \g_0 + [L,R] \cap L_0.
\end{align}
Hence, $y\in \kf_0$ may be written 
$y=x+y'$, where $x \in \g_0$ and $y' \in [L,R] \cap L_0 \subseteq \kf_0$, so $y-y' = x \in \g_0 \cap K=\z$. 
Thus, $x=0$ and $y=y' \in [L,R] \cap L_0 \subseteq [L,R] \subseteq \n$. 
This shows that $\kf_0 \subseteq \n \cap L_0$. For the reverse inclusion, $\n \subseteq K$ in \eqref{LRnK}
implies $\n \cap L_0 \subseteq K \cap L_0 = \kf_0$.
\end{proof}
\begin{Corollary}
\label{Cor:k0}
Let $x \in \kf_0$. Then, $\ad_x$ is nilpotent.
\end{Corollary}
\begin{proof}
From Proposition~\ref{Prop:k0}, $\kf_0 = \n \cap L_0$, and since $\n$ is nilpotent, the result follows.
\end{proof}
\begin{Theorem}
\label{CKZ}
$\Cf = Z(L)$.
\end{Theorem}
\begin{proof}
Let $c\in\Cf$.
By property (iii) in Definition~\ref{Def:normal}, $c$ is semisimple on $L_{i}$, $i \neq 0$. 
As $[c,\g_0] = \z$, $c$ is also semisimple on $\g_0$, and since $L$ is generated by 
$\g_0\oplus\bigoplus_{i \neq 0}L_i$, it follows that every
element of $\Cf$ is semisimple on all of $L$. On the other hand,
$\Cf\subseteq \kf_0$, so by Corollary~\ref{Cor:k0}, $c$ is ad-nilpotent. 
This is only possible if $\ad_c=0$, so $\Cf\subseteq Z(L)$. The reverse inclusion is immediate.
\end{proof}
\begin{Corollary}
\begin{align}
 C_0(\g_0) = Z(\g_0) \oplus Z(L), \qquad
 N_0(\g_0) = \g_0 \oplus Z(L), \qquad
 C(L_s) = Z_s \oplus Z(L).
\end{align}
\end{Corollary}
\begin{proof}
These decompositions follow from Proposition~\ref{Prop:CN} and Theorem~\ref{CKZ}. 
\end{proof}

\noindent
We are now in a position to show that every regular $\Z$-graded Lie algebra is normal.
\begin{Theorem}
$L$ is a normal $\Z$-graded Lie algebra.
\end{Theorem}
\begin{proof}
By construction, $L$ satisfies conditions (i), (iii) and (v) in Definition \ref{Def:normal}.
As pointed out in the Remark following Definition \ref{Def:regular}, $L$ satisfies condition (ii).
By Theorem \ref{CKZ}, $Z(\Cf)\subseteq\Cf=Z(L)$, so condition (iv) is satisfied.
The result now follows.
\end{proof}
\begin{Corollary}
For each index $i$ such that $0 < |i| \leq m$, $[\kf_i,\kf_{-i}] \subseteq Z(L)$.
\end{Corollary}
\begin{proof}
By condition (iv) in Definition~\ref{Def:regular}, Corollary~\ref{Cor:NLLC}, and Theorem~\ref{CKZ}, we have
$[\kf_i,\kf_{-i}] \subseteq N(L_s) \subseteq\Cf= Z(L)$.
\end{proof}

\subsection{Simple roots}

For simplicity, we now assume $L_1$ is irreducible as a $\g_0$-module, so $L_1$ has
a unique lowest weight, $\alpha_0$, and this weight appears with unit multiplicity in $\Phi$.
The set of simple roots is then given by
\begin{align}
 \Pi = \{\alpha_0,\alpha_1,\ldots,\alpha_r\},
\end{align}
where $\alpha_1,\ldots,\alpha_r$ are the simple roots of $\g_0$.
\medskip

\noindent
\textbf{Remark.}
If $L_1$ is not an {\em irreducible} $\g_0$-module, then it is completely reducible (since $L$ is multiplicity free) and $\Pi$ 
contains as many roots as there are summands in the decomposition of $L_1$.
\medskip

\noindent
Corresponding to $\alpha_0$, let $e_0$ be a lowest-weight vector in $L_1$ viewed as a $\g_0$-module.
As shown in Theorem~\ref{Theo:LiLmi} below, $L_{-1}$ is isomorphic to the dual of $L_1$, so there exists
a highest-weight vector $f_0 \in L_{-1}$ of highest weight $-\alpha_0$. 
Correspondingly, we define
\begin{align}
 h_0:=[e_0,f_0].
\end{align}
\begin{Theorem}
\label{Theor:simpler}
For each $\ell=1,\ldots,r$, there exist nonzero $e_\ell\in L_{\alpha_\ell}$, $f_\ell\in L_{-\alpha_\ell}$, 
and $h_\ell\in H$ such that, for all $i,j\in\{0,1,\ldots,r\}$, 
\begin{align}
[e_i,f_j] = {\delta}_{ij}h_i,\qquad [h,e_i] = \alpha_i(h)e_i,\qquad [h,f_i] = -\alpha_i(h)f_i,\qquad \forall h \in H.
\end{align}
\end{Theorem}
\begin{proof}
The result is a consequence of the definition of $e_0,h_0,f_0$ and the fact that $\alpha_1,\ldots,\alpha_r$ 
are the simple roots of $\g_0$.
\end{proof}

\noindent
\textbf{Remark.}
As in the standard theory of semisimple Lie algebras, 
$\langle e_i,h_i,f_i\rangle\cong sl(2,\C)$, $i=1,\ldots,r$.
However, this is not necessarily the case for $i=0$. 
Instead, we note that either $e_0,f_0\in L_s$, in which case 
$h_0\in \h_0$, or $e_0,f_0\in K$, in which case $h_0\in [\kf_1,\kf_{-1}] \subseteq \Cf = Z(L)$.
In either case, we have $h_0 \in H$. 
\begin{Lemma}
\label{h0}
$h_0 \neq 0$.
\end{Lemma}
\begin{proof}
If $L_1 = \g_1 = \s_1 \subseteq L_s$, there is nothing to prove since then $e_0,f_0$ are root vectors of 
weight $\pm\alpha_0$. Thus, suppose $L_1 = \kf_1$ and $[e_0,f_0] = 0$, and 
observe that $e_0 \in \kf_1$ is the lowest-weight vector, while $f_0 \in \kf_{-1}$ is the maximal vector. 
It follows that
\begin{align}
[[e_i,e_0],f_0] = - [[f_0,e_i],e_0] - [[e_0,f_0],e_i] = 0,\qquad \forall i = 1,\ldots,r.
\end{align}
By recursion, $[\kf_1,f_0]=\z$, in contradiction to the reflexivity condition.
\end{proof}
\begin{Proposition}
\label{Prop:slh}
If $\alpha_0(h_0) \neq 0$, then $\langle e_0,h_0,f_0\rangle\cong sl(2,\C)$.
If $\alpha_0(h_0)=0$, then $\langle e_0,h_0,f_0\rangle\cong\h_3$ and $h_0 \in Z(L)$.
\end{Proposition}
\begin{proof}
From Lemma~\ref{h0}, $\dim\langle e_0,h_0,f_0\rangle=3$, where
\begin{align}
[e_0,f_0] = h_0,\qquad [h_0,e_0]=\alpha_0(h_0)e_0,\qquad [h_0,f_0]=-\alpha_0(h_0)f_0.
\end{align}
This establishes the two possible isomorphisms. The result $h_0 \in Z(L)$ follows from the Remark 
preceding Lemma~\ref{h0}.
\end{proof}
\noindent
Both cases in Proposition~\ref{Prop:slh} can occur, as demonstrated in Section~\ref{Sec:Ex}.

\subsection{Structure theory}
\label{Sec:Structure}

\noindent
\textbf{Remark.}
For $e_{\beta} \in L_{\beta}$ and $e_{\beta'} \in L_{\beta'}$, where $\beta,\beta' \in \Phi$ such that 
$\beta+\beta'\neq0$, standard arguments imply that $(e_{\beta},e_{\beta'})=0$.
\begin{Theorem}
\label{Theo:LiLmi}
If $0<|i| \leq m$, then $L_i$ and the dual of $L_{-i}$ are isomorphic as $\g_0$-modules.
\end{Theorem}
\begin{proof}
Let $0 < |i| \leq m$ and $\beta \in \Phi_i$. First suppose $-\beta \notin \Phi_{-i}$. 
Then, $(e_{\beta},L_{-i})=\z$, hence $e_{\beta} \in K$, so
$[e_{\beta},\kf_{-i}]  \subseteq [\kf_{i},\kf_{-i}] \subseteq Z(L)$.
Since the zero weight does not appear in $[e_{\beta},\kf_{-i}]$, it follows that $[e_{\beta},\kf_{-i}] = \z$, 
in contradiction to reflexivity. Hence, the weights in $\Phi_i$ are the negative of the weights in $\Phi_{-i}$.
Second, given an irreducible $\g_0$-module $W \subseteq \kf_i$,
there exists an irreducible $\g_0$-module $W' \subseteq \kf_{-i}$ such that 
$\z \neq [W',W] \subseteq Z(L)$. By Schur's Lemma, this can only be true if $W'$ is isomorphic to the dual of $W$,
so the dual of every irreducible $\g_0$-submodule of $L_i$ must occur in $L_{-i}$. The result now follows.
\end{proof}
\begin{Corollary}
For each index $i$ such that $0<|i| \leq m$, $\Phi_{-i} = -\Phi_i$.
\end{Corollary}
\begin{proof}
The result follows from the proof of Theorem~\ref{Theo:LiLmi}.
\end{proof}

For $\beta \in \Phi$, we define $t_{\beta} \in \h_0$ in the usual way by
\begin{align}
\beta(h) = (t_{\beta},h),\qquad \forall h \in \h_0.
\end{align}
Since the Killing form restricted to $\h_0$ is non-degenerate, this is well-defined. 
\begin{Proposition}
\label{Prop:zinZ}
Let $\beta \in \Phi_i$, $0 < |i| \leq m$, and suppose $x \in L_{\beta}$ and $y \in L_{-\beta}$. 
Then, there exists $z \in Z(L)$ such that
\begin{align}
[x,y] = (x,y)t_{\beta} + z.
\end{align}
\end{Proposition}
\begin{proof}
Observe that $[x,y] \in H$, and let $h \in \h_0$ be arbitrary. Then,
\begin{align}
(h,[x,y]) = ([h,x],y) = \beta(h)(x,y) = (t_{\beta},h)(x,y)=((x,y)t_{\beta},h) = (h,(x,y)t_{\beta}),
\end{align}
so $[x,y] - (x,y)t_{\beta} \in H$ is orthogonal to $\h_0$, hence $[x,y] - (x,y)t_{\beta} \in K \cap H = Z(L)$. 
\end{proof}
\begin{Corollary}
Let $\beta \in \Phi_i$, $0 < |i| \leq m$, and suppose
$x \in L_{\beta}$ and $y \in L_{-\beta}$ such that $(x,y) \neq 0$. Then,
$\langle x,y,[x,y]\rangle\cong sl(2,\C)$.
\end{Corollary}
\begin{proof}
For $(x,y)\neq0$, Proposition~\ref{Prop:zinZ} implies that $L_\beta$ and $L_{-\beta}$ are root spaces for $S$.
It follows that
\begin{align}
 [[x,y],x]=(x,y)\beta(t_\beta)x\neq0,\qquad [[x,y],y]=-(x,y)\beta(t_\beta)y\neq0,
\end{align}
hence the result.
\end{proof}
\begin{Lemma}
\label{Lem:Si}
For every $i$,
\begin{align}
 S_i:=\{x \in L_i\,|\,[x,L_{-i}] \subseteq Z(L)\}
\end{align}
is an $L_0$-submodule of $L_i$.
\end{Lemma}
\begin{proof}
For $y\in L_0$ and $x\in S_i$, we have
\begin{align}
[[y,x],L_{-i}] =[[L_{-i},y],x]+[[x,L_{-i}],y] \subseteq [L_{-i},x]\subseteq Z(L),
\end{align}
so $[y,x]\in S_i$.
\end{proof}
\begin{Proposition}
\label{Prop:xK}
Let $0 \neq x \in L_{\beta}\subseteq L_i$, where $\beta\in\Phi$ and $0 < |i| \leq m$, 
and suppose $(x,L_{-\beta})=\z$. Then, the following holds:
\begin{itemize}
\item[{\rm (i)}]
$x \in K$.
\item[{\rm (ii)}]
$[x,L_{-\beta} \cap \kf_{-i}] \subseteq Z(L)$.
\item[{\rm (iii)}]
If $L_{\beta'} \subseteq \kf_{-i}$ with $\beta + \beta' \neq 0$, then $[x,L_{\beta'}] = \z$.
\item[{\rm (iv)}]
$[x,\kf_{-i}] \subseteq Z(L)$.
\end{itemize} 
\end{Proposition}
\begin{proof}
Under the given assumptions, we have $(x,L_{-i})=\z$, so $x \in K$, which proves part (i). Since $K$ is an 
ideal of $L$, part (ii) follows from the observation that $[x,L_{-\beta}] \subseteq [\kf_{i},\kf_{-i}] \subseteq Z(L)$. As 
to (iii), we have $[x,L_{\beta'}] \subseteq [\kf_{i},\kf_{-i}] \subseteq Z(L)$. Since $\beta + \beta' \neq 0$, 
this can only occur if $[x,L_{\beta'}] = \z$. Finally, part (iv) follows from (ii) and (iii).
\end{proof}
\begin{Corollary}
\label{Cor:LiLmi}
For $0<|i| \leq m$, if $L_i$ is irreducible as an $L_0$-module, then $[L_i,L_{-i}] \subseteq Z(L)$.
\end{Corollary}
\begin{proof}
This follows from Lemma~\ref{Lem:Si} and Proposition~\ref{Prop:xK} (iv).
\end{proof}

\noindent
It follows that for $L_\beta\subseteq L_i$, where $0<|i| \leq m$, the root-space generators 
$e_{\beta} \in L_{\beta}$ and $f_{\beta} \in L_{-\beta}$ either generate an $sl(2,\C)$ 
subalgebra or $[e_{\beta},f_{\beta}] \in Z(L)$. In the former case, the generators may be normalised
such that 
\begin{align}
h_{\beta}:=[e_{\beta},f_{\beta}]= \frac{2}{(\beta,\beta)} t_{\beta} + z,\qquad z \in Z(L),
\end{align}
while in the latter case, we have the following result.
\begin{Proposition}
\label{Prop:yHeis}
Let $0 \neq x \in L_{\beta} \cap L_i$, where $\beta\in\Phi$ and $0 < |i| \leq m$, and suppose $(x,L_{-\beta})=\z$. 
Then, there exist $y \in L_{-\beta} \cap \kf_{-i}$ and $z\in Z(L)$ such that $\langle x,y,z\rangle\cong\h_3$.
\end{Proposition}
\begin{proof}
If $0 \neq x \in L_{\beta}$, then there exists $y \in L_{-\beta}$ such that $[x,y] \neq 0$. 
Otherwise, by Proposition~\ref{Prop:xK}, we would have $x \in \kf_{i}$ with $[x,\kf_{-i}] = \z$, in contradiction with
reflexivity. The result now follows from Propositions~\ref{Prop:zinZ} and~\ref{Prop:xK}.
\end{proof}
\begin{Corollary} 
\label{Cor:Liirred}
Let $L_i\neq\z$, $0 < |i| \leq m$, be irreducible as an $L_0$-module. Then, 
\begin{itemize}
\item[{\rm (i)}]
$L_i \oplus L_{-i} \subseteq K$;
\item[{\rm (ii)}]
$[L_i,L_{-i}] = \C z$ for some nonzero $z\in Z(L)$.
\end{itemize}
\end{Corollary}
\begin{proof}
By Proposition~\ref{Prop:xK}, we have $(x,L_{-i}) = \z$, so $x \in K$. Irreducibility then implies that 
$\z \neq K \cap L_i = L_i$, so $L_i \subseteq K$, which is enough to prove part (i). As to part (ii), let $\beta_0 \in \Phi_i$ be 
the lowest weight of $L_i$ (occurring with unit multiplicity in $\Phi_i$), $e_0 \in L_i$ a corresponding
lowest-weight vector, and $f_0 \in L_{-i}$ a corresponding highest-weight vector. Then, by Proposition~\ref{Prop:yHeis}, 
$[e_0,f_0] = z$ for some $0 \neq z \in Z(L)$. Hence, by Proposition~\ref{Prop:xK} (iii), 
$[e_0,L_{-i}] \subseteq \C z$. But $\{x \in L_i\,|\,[x,L_{-i}] \subseteq \C z\}$ is 
an $L_0$-submodule of $L_i$ since, for $s\in\{x \in L_i\,|\,[x,L_{-i}] \subseteq \C z\}$ and $x \in L_0$,
\begin{align}
 [[x,s],L_{-i}]=-[[L_{-i},x],s]-[[s,L_{-i}],x]\subseteq\C z.
\end{align}
\end{proof}

\noindent
The $\g_0$-module decomposition \eqref{Li} induces the following partitioning:
\begin{align}
\Phi_i = \Phi_i^\s \cup \Phi_i^\kf,\qquad\forall i.
\end{align}
Thus, in the case $|i| > m$, we have $L_i = \kf_i$ and $\Phi_i = \Phi_i^\kf$, while in the case $i = 0$, 
we have $L_0 = \g_0 \oplus \kf_0$ and $\Phi_0 = \Phi_0^\s \cup \Phi_0^\kf$. We set
\begin{align}
\Phi_\s:=\bigcup_i\Phi_i^\s,\qquad \Phi_\kf:=\bigcup_i\Phi_i^\kf,
\end{align}
so that $\Phi = \Phi_\s \cup \Phi_\kf$. 
\begin{Lemma}
Let $x \in \g_i$, $0 < |i| \leq m$. Then, $(x,\g_{-i})=\z$ implies $x=0$.
\end{Lemma}
\begin{proof}
Let $x \in \g_i$. Since the Killing form restricted to $L_s$ is non-degenerate, we have
$(x,\g_{-i})=\z$, so $(x,L_s)=\z$, hence $x = 0$.
\end{proof}

\noindent
We thus arrive at the following analogue of Proposition~\ref{Prop:yHeis}.
\begin{Proposition}
Let $0 \neq x \in L_{\beta}\cap\g_i$, where $\beta\in\Phi$ and $0 < |i| \leq m$.
Then, there exists $y \in L_{-\beta} \cap \g_{-i}$ such that $(x,y) \neq 0$
and $\langle x,y, [x,y]\rangle\cong sl(2,\C)$. 
\end{Proposition}
\begin{proof}
If no such $y$ exists, then $\z=(x,L_{-\beta} \cap \g_{-i}) = (x,\g_{-i})$, hence $x = 0$, a contradiction.
\end{proof}
\begin{Proposition}
\label{Prop:Ligi}
\mbox{}
\begin{itemize}
\item[{\rm (i)}]
If $L_i = \g_i$, then  $[L_j,[\kf_{-j},L_{\pm i}]] = [\kf_j,L_{-i-j}] =\z$.
\item[{\rm (ii)}]
If $i+j \neq 0$ and $L_{i+j} = \g_{i+j}$, then $[\kf_i,L_j] = [\kf_j,L_i] = \z$.
\item[{\rm (iii)}]
Let nonzero indices $i,j$ satisfy $|i|,|j|\leq m$ with $L_i = \g_i$, $L_j = \g_j$ and $L_{i+j} = \kf_{i+j}$. 
Then,
\begin{align}
[L_i,L_{-i-j}] = [L_j,L_{-i-j}] = [L_i,L_j] = \z.
\end{align}
\end{itemize}
\end{Proposition}
\begin{proof}
For (i), it suffices to observe that $[\kf_{-j},L_{\pm i}]\subseteq K \cap L_{\pm i} = \z$ and similarly 
$[\kf_j,L_{-i-j}] \subseteq K \cap L_{\pm i} = \z$. Part (ii) follows from the fact that $[\kf_i,L_j]$ and $[\kf_j,L_i]$ are 
subsets of $K \cap L_{i+j} = \z$. As to (iii), we have $[L_i,L_{-i-j}] \subseteq K \cap L_j = \z$, 
and similarly for $[L_j,L_{-i-j}]$.  It follows that
\begin{align}
 [[L_i,L_j],L_{-i-j}] \subseteq [[L_{-i-j},L_i],L_j] + [[L_j,L_{-i-j}],L_i] = \z.
\end{align}
Reflexivity then implies that $[L_i,L_j] = \z$.
\end{proof}
\begin{Corollary}
\mbox{} 
\begin{itemize}
\item[{\rm (i)}]
If $L_i = \g_i$ and $L_j = \kf_j$, then $[L_j,[L_{-j},L_{\pm i}]] = [L_j,L_{-i-j}] =\z$.
\item[{\rm (ii)}]
If $i+j \neq 0$ and $L_{i+j} = \g_{i+j}$, then $[L_i,L_j] = [\g_i,\g_j]$.
\end{itemize}
\end{Corollary}
\begin{proof}
The results follow immediately from Proposition~\ref{Prop:Ligi} (i) and (ii), 
noting that $L_i=\g_i\oplus\kf_i$ for all $i$, cf.~\eqref{Li}.
\end{proof}
\begin{Proposition}
\mbox{}
\begin{itemize}
\item[{\rm (i)}] 
If $(\beta,\beta') < 0$ for some $\beta \in \Phi_i^s$ and $\beta' \in \Phi_j$, then $[\g_i,L_j] \neq \z$.
\item[{\rm (ii)}] 
If $(\beta,\beta') > 0$ for some $\beta \in \Phi_i^s$ and $\beta' \in \Phi_j$, then $[\g_{-i},L_j] \neq \z$.
\end{itemize}
\end{Proposition}
\begin{proof}
Part (i) follows from the representation theory of the $sl(2,\C)$ subalgebra generated by nonzero elements 
$e_\beta\in L_\beta$ and $e_{\beta'}\in L_{\beta'}$.
Indeed, if $[\g_i,L_j] = \z$, then all root vectors in $L_j$ are maximal for $sl(2,\C)$, so
$(\beta,\beta') > 0$ for all $\beta' \in \Phi_j$, a contradiction. Part (ii) follows similarly.
\end{proof}
\noindent
\textbf{Remark.}
Let $L$ be a simple Lie algebra and $L_0 \subset L$ the reductive Lie subalgebra obtained by 
omitting the $s$th node of the Dynkin diagram of $L$. Then, $L$ admits a $\Z$-graded decomposition
\begin{align}
L=\bigoplus_{i=-k}^{k}L_i,
\end{align}
where $k$ is the coefficient of the simple root $\alpha_s$ in the decomposition of the highest root of $L$ into a 
sum of simple roots. With this structure, $L$ is a regular $\Z$-graded Lie algebra. 
For a general simple Lie algebra, $k \leq 6$, while for a classical simple Lie algebra, 
$k \leq 2$. The simple roots $\alpha_j$, $j\neq s$, are the simple roots of $L_0$, 
while $\alpha_s$ is the lowest weight of the $L_0$-module $L_1$. 
The corresponding generators $e_s,h_s,f_s$ generate an $sl(2,\C)$ subalgebra.

\section{Irreducible algebras}
\label{Sec:Connected}

In this section, all $\Z$-graded Lie algebras are assumed regular.
\begin{Definition}
\label{Def:connected}
A regular $\Z$-graded Lie algebra $L$ is said to be
$$
\begin{array}{ll}
\mbox{\em connected} & 
\mbox{if \
$\left\{\begin{array}{ll} \!\!\![L_1,L_i]=L_{i+1}\mbox{ or }[L_{-1},L_{i+1}]=L_i, &\forall i>0,\\[.2cm] 
 \!\!\![L_{-1},L_i] = L_{i-1}\mbox{ or }[L_1,L_{i-1}] = L_i, &\forall i<0; \end{array}\right.$}
\\[.6cm]
\mbox{\em transitive} & 
\mbox{if \
$\left\{\begin{array}{ll} \!\!\![L_{-1},L_{i+1}]=L_i \mbox{ for all $i>0$ such that $L_{i+1}\neq\z$},\\[.2cm] 
 \!\!\![L_1,L_{i-1}]=L_i \mbox{ for all $i<0$ such that $L_{i-1}\neq\z$}; \end{array}\right.$}
\\[.6cm]
\mbox{\em strongly graded} & 
\mbox{if it is generated by the subspace $L_{-1} \oplus L_0 \oplus L_1$;}
\\[.3cm]
\mbox{\em irreducible} & 
\mbox{if, for each $i\neq0$, $L_i$ is an irreducible $\g_0$-module.}
\end{array}
$$
\end{Definition}
\begin{Definition}
\label{Def:stype}
Let the index $i$ satisfy $|i| \leq m$. 
If $L_i = \g_i$, then $i$ is said to be of semisimple type $(\sr$-type$)$. 
If $L_i = \kf_i$, then $i$ is said to be of Heisenberg type $(\hr$-type$)$.
\end{Definition}
\noindent
\textbf{Remark.} 
Every index $i$, $0<|i| \leq m$, for which $L_i$ is irreducible as a $\g_0$-module, 
is either of $\sr$- or $\hr$-type.
\medskip

\noindent
\textbf{Remark.} 
If $L$ is strongly graded, then there are two possibilities:
(i)  All nonzero indices are of $\hr$-type, and (ii) all indices $i$ with $|i| \leq m$ are of $\sr$-type.
In \cite{Kac68}, $L$ is assumed strongly graded.
\medskip

As illustrated in Section~\ref{Sec:Ex}, many $\Z$-graded Lie algebras of known physical interest are irreducible in the sense
of Definition~\ref{Def:connected}. Accordingly, in the remainder of this section, we shall assume that $L$ is irreducible.
\begin{Proposition}
Let index $i$, $0 < i < m$, be of $\hr$-type. Then,
\begin{align}
[L_i,L_i] = [L_{-i},L_{-i}] = \z.
\end{align}
\end{Proposition}
\begin{proof}
If $[L_i,L_i]\neq\z$, then $[L_i,L_i] = L_{2i}$, by irreducibility. 
By Corollary~\ref{Cor:Liirred}, $[L_{2i},L_{-2i}] \neq\z$, 
so $[L_{-i},L_{-i}]=L_{-2i}$, again by irreducibility.
Similarly, if $[L_{-i},L_{-i}]\neq\z$, then $[L_{-i},L_{-i}]=L_{-2i}$ and $[L_i,L_i] = L_{2i}$.
With either assumption, Corollary~\ref{Cor:LiLmi} then implies
\begin{align}
[L_{-i},L_{2i}] = [[L_{-i},[L_i,L_i]] = [[L_i,[L_{-i},L_i]] +[[L_i,[L_i,L_{-i}]] = \z,
\end{align}
hence
\begin{align}
[L_{-2i},L_{2i}] = [[L_{-i},L_{-i}],L_{2i}] = [[L_{2i},L_{-i}],L_{-i}] + [[L_{-i},L_{2i}],L_{-i}] = \z,
\end{align}
a contradiction.
\end{proof}
\begin{Definition}
$L$ is said to be symmetric if it is balanced and satisfies
\begin{align}
[L_i,L_j] = \z\quad\Longleftrightarrow\quad [L_{-i},L_{-j}] = \z. 
\end{align}
\end{Definition}
\noindent
\textbf{Remark.} 
For $L$ symmetric, we may assume a $\Z$-graded decomposition of the form
\begin{align}
 L=\bigoplus_{i=-m}^mL_i.
\end{align}
\begin{Corollary}
Let $L$ be symmetric. For each $i>0$, there exists $z_i\in L_0$ such that
\begin{align}
 L_{-i} \oplus {\C}z_i \oplus L_i\cong\h_{n_i},\qquad n_i=2\dim L_i+1.
\end{align}
\end{Corollary}
\begin{proof}
The result is a consequence of Corollary~\ref{Cor:Liirred} (ii).
\end{proof}

\subsection{Transitive algebras}

We now assume that the irreducible $\Z$-graded Lie algebra $L$ is transitive. 
It is convenient to consider the unbalanced ($k \neq l$) and balanced ($k = l$) cases separately.

In the unbalanced case, all indices $i$ with $|i| >m$ are of $\hr$-type. 
As $L$ is assumed transitive and $K$ is an ideal, it follows that all nonzero indices are of $\hr$-type. 
Hence,
\begin{align}
K = \kf_0\oplus\bigoplus_{i \neq 0} L_i,\qquad L_s = \g_0.
\end{align}

In the balanced case, there are several possible subcases.
First, if all indices are of $\sr$-type, then $L$ is a reductive Lie algebra with $K = Z(L)$. 
In fact, $\g_0$ is obtained by removal of a node from the Dynkin diagram of $L_s$. This imposes the 
constraint $k\leq 6$.

Second, if not all indices are of $\sr$-type, then there exists an $\hr$-type index $j$, $0 < j \leq k$. 
By the same argument as above, all indices $i$ such that $0 < i \leq j$ are also of $\hr$-type. 
By the transitivity condition $L_i = [L_{-1},L_{i+1}]$, this would then imply index $i = 1$ is of $\hr$-type 
and hence any index $i$ such that $1\leq i<k$ is of $\hr$-type. 
This means that every nonzero index is of $\hr$-type, or
there exists exactly one $\sr$-type index $k$ while all remaining indices are of $\hr$-type. 
In the latter situation,
\begin{align}
 K=L_{-k+1} \oplus\cdots\oplus L_{-1} \oplus \kf_0 \oplus L_1 \oplus\cdots\oplus L_{k-1},
\end{align}
and we have the $\g_0$-module direct-sum decomposition
\begin{align}
L_s = L_{-k} \oplus \g_0 \oplus L_k.
\end{align}
Moreover, $(L_s,\g_0)$ corresponds to a Hermitian symmetric pair.

\subsection{Connected algebras}

We now assume that the irreducible $\Z$-graded Lie algebra $L$ is connected. The irreducibility then implies that
\begin{align}
 [L_{1},L_i] = L_{i+1}, \qquad [L_{-1},L_{-i}] = L_{-i-1},\qquad i>0,
\end{align}
which is a strengthened version of the connectivity condition in Definition~\ref{Def:connected}.
\begin{Proposition}
\label{Prop:stype}
If index $i=1$ is of $\sr$-type, then every nonzero index $i$ with $|i| \leq m$ is of $\sr$-type. 
\end{Proposition}
\begin{proof}
Assume that index $i=2$ is of $\hr$-type. 
By the connectivity condition, $[L_1,L_1] = L_2$ or $[L_{-1},L_2] = L_1$. 
The first possibility contradicts Proposition~\ref{Prop:Ligi}, 
while the second possibility would imply $L_1 \subseteq K$, again a contradiction.
Hence, index $i=2$ is of $\sr$-type. Proceeding, we now assume that $i$, $2 < i \leq m$, is the smallest index 
of $\hr$-type. By the connectivity condition, $[L_1,L_{i-1}] = L_i$ or $[L_{-1},L_i] = L_{i-1}$.
The first possibility contradicts Proposition~\ref{Prop:Ligi}, while the second possibility would imply 
$L_{i-1} \subseteq K$, in contradiction to the minimality of $i$. In conclusion, index $i$ is not of $\hr$-type.
\end{proof}
\noindent
\textbf{Remark.} 
In the case of Proposition~\ref{Prop:stype}, we have
\begin{align}
 K = \kf_0\oplus\bigoplus_{|i| >m} L_i
\end{align}
and the $\g_0$-module decomposition
\begin{align}
 L_s = L_{-m} \oplus\cdots\oplus L_{-1} \oplus \g_0 \oplus L_1 \oplus\cdots\oplus L_m.
\label{Lsm}
\end{align}
In the case the $\g_0$-module $L_1$ is irreducible, this imposes the constraint $m\leq 6$,
while $m\leq 2$ if $L_s$ is a classical Lie algebra.

In the general case, we have
\begin{align}
 K =\kf_{-1} \oplus \kf_0 \oplus \kf_1\oplus \bigoplus_{|i| > 1}^{\hr\text{-type}} L_i
\end{align}
and the $\g_0$-module decomposition
\begin{align}
 L_s = \g_{-1} \oplus \g_0  \oplus \g_1\oplus \bigoplus_{|i|>0}^{\sr\text{-type}} L_i,
\end{align}
where the $i$-sums are over all $\hr$-type, respectively $\sr$-type, indices. 
Since this decomposition arises by removal of a node in the Dynkin diagram of a reductive Lie algebra,
this imposes strong constraints on the distribution of s-nodes:
Either every nonzero index $i$ is of $\hr$-type, as discussed above, or there 
exists a minimal s-index $p<m$ and a positive integer $n$ such that $0<pn\leq m$ and such that the indices 
$i = 0,\pm p,\ldots,\pm np$ are of $\sr$-type while the remaining indices are of $\hr$-type. In the latter case,
\begin{align}
K = \kf_0\oplus\bigoplus_{i \neq 0,\pm p,\pm 2p,\ldots,\pm np}\kf_i,
\end{align}
and we have the $\g_0$-module decomposition
\begin{align}
L_s = L_{-np}\oplus\cdots\oplus L_{-p} \oplus \g_0 \oplus L_p\oplus\cdots\oplus L_{np}.
\label{L1Li}
\end{align}
This imposes the constraint $n\leq 6$.

The case $p=2$ is of particular interest.
\begin{Proposition}
Suppose all even indices $i$ with $0<|i|\leq m$ are of $\sr$-type with all remaining indices of $\hr$-type. 
Then, there exists nonzero $z \in Z(L)$ such that for any $\hr$-type index $i$ with $|i| \leq m$, 
we have $[L_i,L_{-i}]=\C z$.
\end{Proposition}
\begin{proof}
By assumption, index $2$ is of $\sr$-type and index 1 is of $\hr$-type. 
By Corollary~\ref{Cor:Liirred}, we may then write 
\begin{align}
[L_1,L_{-1}] = \C z,
\end{align}
for some nonzero $z \in Z(L)$. 
If index $i$ is of $\sr$-type, then indices $i \pm 1$ are of $\hr$-type and by the connectivity 
condition in Definition~\ref{Def:connected}, we have $L_{i+1} = [L_1,L_i]$ or $L_i = [L_{-1},L_{i+1}]$. The latter 
case cannot occur, since $i$ is of $\sr$-type, so $L_{i+1} = [L_1,L_i]$. Similarly, $L_{i-1} = [L_{-1},L_i]$, so
\begin{align}
[L_{i+1},L_{-i-1}] = [L_{i+1},[L_{-1},L_{-i}]] = [L_{-i},[L_{i+1},L_{-1}]] + [L_{-1},[L_{-i},L_{i+1}]] = [L_{-1},L_1],
\end{align}
where the last equality follows from Proposition~\ref{Prop:Ligi}. 
We thus obtain $[L_{i+1},L_{-i-1}] = \C z$ and similarly $[L_{i-1},L_{-i+1}] = \C z$, thereby concluding the proof.
\end{proof}

\section{Examples}
\label{Sec:Ex}

We present four classes of examples: semisimple Lie algebras and their Borel subalgebras, conformal Galilei 
algebras, including Schr\"odinger algebras, model filiform Lie algebras, and (extended) Heisenberg algebras.
The extended Heisenberg and semisimple Lie algebras illustrate that a Lie algebra may admit several inequivalent $\Z$-gradings, 
and the conformal Galilei algebras illustrate that roots may occur with nontrivial multiplicity.
In certain cases, the conformal Galilei and extended Heisenberg algebras also provide examples 
of non-regular yet normal $\Z$-graded Lie algebras.

\subsection{Semisimple Lie algebras}
\label{Sec:semi}

Let $L$ be a semisimple Lie algebra with fundamental system $\Pi=\{\alpha_1,\ldots,\alpha_r\}$,
Cartan matrix elements $A_{ij}=(\alpha_i^\vee,\alpha_j)$, $i,j=1,\ldots,r$,
and triangular decomposition $L=L_-\oplus\h\oplus L_+$, 
where $\h$ denotes a CSA with Chevalley basis $\{h_1,\ldots,h_r\}$, while
$L_-=\spa\{f_\alpha\,|\,\alpha\in\Phi^+\}$ and $L_+=\spa\{e_\alpha\,|\,\alpha\in\Phi^+\}$.
For each $\alpha\in\Phi^+$, there exist unique $a_1,\ldots,a_r\in\Nb_0$ such that
\begin{align}
 \alpha=\sum_{i=1}^ra_i\alpha_i.
\end{align}

Taking the Cartan generator
\begin{align}\label{dAinv}
 d=\sum_{i,j=1}^r(A^{-1})_{ij}h_j
\end{align}
as the level operator, we get the $\Z$-grading ($i\in\Nb$)
\begin{align}
 L_{-i}=\spa\{f_\alpha\,|\,\het(\alpha)=i,\,\alpha\in\Phi^+\},\quad\
 L_0=\h,\quad\
 L_i=\spa\{e_\alpha\,|\,\het(\alpha)=i,\,\alpha\in\Phi^+\},
\end{align}
where the \textit{height} of $\alpha\in\Phi^+$ is defined as $\het(\alpha):=\sum_{i=1}^ra_i$ and arises as
\begin{align}
 [d,e_\alpha]=\het(\alpha)e_\alpha,\qquad
 [d,f_\alpha]=-\het(\alpha)f_\alpha.
\end{align}
This $\Z$-grading of $L$ is balanced (see Definition~\ref{Def:Zgraded}) 
and regular (see Definition~\ref{Def:regular}), as well as
connected, transitive, strongly graded and irreducible (see Definition~\ref{Def:connected}).

To illustrate that a given Lie algebra may admit several inequivalent $\Z$-gradings 
(see Remark following Definition~\ref{Def:Zgraded}), let us consider the Borel subalgebra of $L$ given by $B=\h\oplus L_+$.
As $d$ in \eqref{dAinv} is an element of $B$, it generates a $\Z$-grading of $B$ that is inherited from the one of $L$:
$B=L_0\oplus\bigoplus_{i\in\Nb}L_i$. Although this $\Z$-grading is `highly unbalanced', taking instead $h_1-h_r$ as
the level operator yields a nontrivial balanced $\Z$-grading of the Borel subalgebra $B$ of $sl(r+1)$ for $r>1$.
To see this, recall that the set of positive roots of $sl(r+1)$ is given by
\begin{align}
 \Phi^+=\Pi\cup\{\alpha_{ij}\,|\,1\leq i<j\leq r\},\qquad 
 \alpha_{ij}:=\alpha_i+\cdots+\alpha_j,
\end{align}
and introduce the shorthand notation $e_i=e_{\alpha_i}$ and $e_{ij}=e_{\alpha_{ij}}$.
For $r\geq4$, we thus have
\begin{align}
 B=B_{-2}\oplus B_{-1}\oplus B_0\oplus B_1\oplus B_2,
\end{align}
where
\begin{gather}
 B_{-2}=\spa\{e_r,e_{2r}\},\qquad
 B_{-1}=\spa\{e_2,e_{ir},e_{2j}\,|\,i=3,\ldots,r-1;\,j=3,\ldots,r-2\},
 \\[.1cm]
 B_0=\h\oplus\spa\{e_2,\ldots,e_{r-2},e_{1r},e_{2,r-1},e_{ij}\,|\,3\leq i<j\leq r-2\},
 \\[.1cm]
 B_1=\spa\{e_{r-1},e_{i,r-1},e_{1j}\,|\,i=3,\ldots,r-2;\,j=2,\ldots,r-2\},\quad
 B_2=\spa\{e_1,e_{1,r-1}\}.
\end{gather}
For $r=3$, we take $\frac{1}{2}(h_1-h_3)$ as the level operator and get the $\Z$-grading
\begin{align}\label{BBB}
 B=B_{-1}\oplus B_0\oplus B_1,
\end{align}
where
\begin{align}
 B_{-1}=\spa\{e_3,e_{23}\},\qquad
 B_0=\h\oplus\spa\{e_2,e_{13}\},\qquad
 B_1=\spa\{e_1,e_{12}\}.
\end{align}
Finally for $r=2$, we take $\frac{1}{3}(h_1-h_2)$ as the level operator and again get a $\Z$-grading
of the form \eqref{BBB} but this time with
\begin{align}
 B_{-1}=\spa\{e_2\},\qquad
 B_0=\h\oplus\spa\{e_{12}\},\qquad
 B_1=\spa\{e_1\}.
\end{align}

Still for $r=2$, we may alternatively take $h_1$ as the level operator, in which case we get the unbalanced $\Z$-grading
\begin{align}
 B=B_{-1}\oplus B_0\oplus B_1\oplus B_2,
\end{align}
where
\begin{align}
 B_{-1}=\spa\{e_2\},\qquad
 B_0=\h,\qquad
 B_1=\spa\{e_{12}\},\qquad
 B_2=\spa\{e_1\}.
\end{align}
As this example fails to be reflexive, it is not regular (see Definition~\ref{Def:regular}).

\subsection{Conformal Galilei algebras}

For each $n\in\Nb$ and $\ell\in\frac12\Nb=\{\frac12,1,\frac32,2,\ldots\}$, 
let $g_\ell(n)$ denote the \textit{conformal Galilei (Lie) algebra} with basis
\begin{align}
 \Bf_\ell(n)=\{D,H,C,J_{ij},P_{m,i}\,| \, J_{ij}=-J_{ji};
\, i,j =1, \ldots, n;
\, m = 0,1, \ldots, 2\ell 
\}
\end{align}
and corresponding nonzero Lie products
\begin{gather}
[D,H] = 2H,   \qquad 
[D,C] = -2C, \qquad 
[C, H] = D, 
\\[.1cm]
[H,P_{m,i}] = -mP_ {m-1,i}, \qquad  
[D,P_{m,i}] =2(\ell - m)P_ {m,i}, \qquad 
[C,P_{m,i}] =(2\ell - m)P_{m+1,i},
\\[.1cm]
[J_{ij}, J_{k\ell}] =\delta_{ik}J_{j\ell} + \delta_{j \ell}J_{ik} - \delta_{i \ell}J_{jk}
- \delta_{jk}J_{i \ell},\qquad
[J_{ij},P_{m,k} ] = \delta_{ik}P_{m,j} - \delta_{jk}P_{m,i}.
\end{gather}
It follows that
\begin{align}
 g_\ell(n)=so(2,1)\oplus so(n)\oplus\bigoplus_{m=0}^{2\ell}\mathsf{P}_m,
\end{align}
where
\begin{gather}
 so(2,1)=\spa\{D,H,C\},\qquad
 so(n)=\spa\{J_{ij}\,|\,1\leq i<j\leq n\},
 \\[.1cm]
 \mathsf{P}_m:=\spa\{ P_{m,i}\,|\,i=1,\ldots, n\},\qquad m=0,1,\ldots,2\ell.
\end{gather}
Writing $L=g_\ell(n)$, we see that $L_s=so(2,1)\oplus so(n)$ is the semisimple Levi factor $S$, 
and that the radical and the Killing radical coincide and are given by
\begin{align}
 R=K=\bigoplus_{m=0}^{2\ell}\mathsf{P}_m.
\end{align}
\noindent
\textbf{Remark.}
The notation $so(2,1)$ and $so(n)$ reflects the physical origin of the algebras.
Here, we are considering their complexifications.
\medskip

For $\ell$ half-odd integer ($\ell\in\{\frac12,\frac32,\frac52,\ldots\}$), 
$g_\ell(n)$ admits a central extension $g_\ell(n)\to \hat{g}_\ell(n)=g_\ell(n)\oplus\C M$ (see, e.g.,~\cite{GM11}), 
enlarging the basis to $\hat{\Bf}_\ell=\Bf_\ell\cup\{M\}$, and with $M$ arising in
\begin{align}
 [P_{m,i},P_{r,j}]= \delta _{ij}  \delta _{m+r,2\ell}  (-1)^{m+\ell+\frac12}  (2\ell -m)! m! M.
\end{align}
The radical and the Killing radical again coincide, and are now given by
\begin{align}
 R=K=\bigoplus_{m=0}^{2\ell}\mathsf{P}_m\oplus\C M.
\end{align}
\noindent
\textbf{Remark.}
For each $n\in\Nb$, the \textit{Schr{\"o}dinger algebra $S(n)$} \cite{Nie72} is isomorphic to $\hat{g}_\frac12(n)$
and is related to the symmetries of the free Schr{\"o}dinger equation in $(n+1)$-dimensional space-time. 
\medskip

For $\ell$ half-odd integer ($\ell\in\{\frac12,\frac32,\frac52,\ldots\}$), we consider $L=\hat{g}_\ell(n)$ and
take $d =-D$ as the level operator, thereby getting the $\Z$-grading
\begin{align}
 L = \bigoplus_{i=-2\ell}^{2\ell}L_i, 
\end{align}
where $L_i= \mathsf{P}_{\ell+\frac{i}{2}}$ for $i$ odd, while
\begin{align}
L_{-2}=\C H,\qquad
L_{0}= \C D \oplus so(n)\oplus\C M,\qquad 
L_2=\C C,\qquad
L_{2j}=\z\ \ \text{for}\ \ |j|>1.
\end{align}
It follows that the $\Z$-grading is balanced with $k=l=2\ell$,
and that $\g_{\pm2}=L_{\pm2}$, $\g_0 =\C D \oplus so(n)$, $\kf_0 = \C M$, $\kf_i = L_i$ for $i$ odd,
and that the $\Z$-grading is regular. 
Without the central extention, the reflexivity condition (iii) in Definition \ref{Def:regular} is not satisfied, 
so the corresponding $\Z$-graded algebra is non-regular, albeit normal.

For $\ell$ integer ($\ell\in\{1,2,3,\ldots\}$), we consider $L=g_\ell(n)$ and take
$d =-\frac{1}{2}D$ as the level operator, thereby getting the $\Z$-grading
\begin{align}
L = \bigoplus_{i=-\ell}^{\ell} L_i,
\end{align}
where $L_i = \mathsf{P}_{\ell+i}$, for $|i| > 1$, while 
\begin{align}
 L_{-1} = \C H\oplus \mathsf{P}_{\ell-1},\qquad
 L_0 = \C D\oplus so(n)\oplus \mathsf{P}_{\ell},\qquad
 L_1 = \C C\oplus \mathsf{P}_{\ell+1}.
\end{align}
It follows that the $\Z$-grading is balanced with $k=l=\ell$,
and that $\g_0 =\C D\oplus so(n)$, $\kf_i = \mathsf{P}_{\ell+i}$ for all $i$,
and $L_s = \g_{-1} \oplus \g_0 \oplus \g_1$ with $\g_{-1}=\C H$ and $\g_1 = \C C$. 
With reference to \eqref{m0} in Section~\ref{Sec:L0} below, we see that
$\m_0 =\mathsf{P}_{\ell}$, so $\kf_0 = \m_0\oplus Z(L)$. 
Indeed, there are nonzero root spaces in $\m_0$ (for $n > 2$). Moreover, for $n$ odd, the roots in $\m_0$ 
all occur in $\g_0$, thus providing an example with roots of multiplicity 2.
Moreover, the reflexivity condition (iii) in Definition \ref{Def:regular} is not satisfied, 
so the $\Z$-graded algebra $L$ is non-regular, albeit normal.

\subsection{Model filiform Lie algebras}
\label{Sec:Filiform}

A finite-dimensional nilpotent Lie algebra with maximal nilindex is called a filiform Lie algebra \cite{Vergne70}; 
see also \cite{Bla58}. 
The smallest set of such algebras from which all other filiform Lie algebras can be described as linear
deformations \cite{GK96}, comprises the so-called model filiform Lie algebras.
For each $n\geq3$, we denote by $F_n$ the \textit{model filiform Lie algebra} with basis 
$\beta_n=\{x_{-1}, x_0,x_1, \ldots, x_{n-2}\}$ and corresponding nonzero Lie products
\begin{align}
 [x_\ell,x_{-1}] = x_{\ell-1},\qquad 
 \ell=1,\ldots, n-2.
\end{align}
We enlarge $F_n$ to $L:=F_n\oplus\C d$, setting
\begin{align}
  [d,x_j]=jx_j,\qquad 
  j=-1,0,1,\ldots, n-2,
\end{align}
whereby $L$ becomes a Lie algebra with basis $\beta_n\cup \{d\}$.
The Killing radical is given by $K=L'=F_n$, the radical by $R=L$, 
the centre by $Z(L)=\C x_0$, while $L_s=Z_s=\C d$.
From
\begin{align}
 [d,[x_\ell,x_{-1}]]=[[d,x_\ell],x_{-1}]+[x_\ell,[d,x_{-1}]],\qquad \ell=1,\ldots,n-2,
\end{align}
it follows that $d$ is an inner derivation of $L$.

Taking the level operator to be $d$, we get the $\Z$-grading
\begin{align}
 L=L_{-1}\oplus L_0\oplus L_1\oplus \cdots\oplus L_{n-2},
\end{align}
where
\begin{align}
 L_0 = \C x_0\oplus\C d,\qquad
 L_i = \C x_i,\qquad i=-1,1,2,\ldots,n-2.
\end{align}
This $\Z$-grading of $L$ is regular, irreducible, connected, and transitive.
Only for $n=3$ is it balanced and strongly graded.

\subsection{Extended Heisenberg algebras}

Consider the finite-dimensional \textit{Heisenberg (Lie) algebra $H_n$} with basis 
$\beta_n=\{ b_j,b_j^\dagger,c\,|\, j=1,\ldots,n\}$ and Lie products
\begin{align}
 [c,b_j]=0=[c,b^\dagger_j],\qquad
 [b_i,b_j]=0=[b_i^\dagger,b_j^\dagger],\qquad 
 [b_i,b_j^\dagger]=\delta_{ij}c,\qquad i,j\in\{1,\ldots,n\}.
\end{align}
Note that $H_1$ is isomorphic to the model filiform $F_3$ from Section~\ref{Sec:Filiform}.

We enlarge $H_n$ to the \textit{extended Heisenberg algebra} $L=L^{(\lambda,\mu)}:=H_n\oplus\C d$, 
setting ($\lambda\in\Z^n$, $\mu\in\Z$)
\begin{align}
 [d,c]=\mu c,\qquad
 [d,b_j]=(\mu-\lambda_j)b_j,\qquad 
 [d,b^\dagger_j]=\lambda_jb^\dagger_j,\qquad 
 j=1,\ldots, n,
\end{align}
thereby turning $L$ into a Lie algebra with basis $\beta_n\cup \{d\}$ and $d$ an inner derivation.
It follows that $K=H_n$, $R=L$, $S=\{0\}$, $L_s=Z_s=\C d$, and
\begin{align}
 Z(L)=\begin{cases} \C c\oplus\C d,\ &\mu=0,\ \lambda=0,\\
  \C c,\ &\mu=0,\ \lambda\neq0,\\
  \{0\},\ &\mu\neq0.
  \end{cases}
 \end{align}
For simplicity, we now set $\mu=0$ and restrict $\lambda$ to $\Nb_0^n$.
In this case, $L^{(\lambda)}\cong L^{(\lambda')}$ if $\lambda$ or $\lambda'$ is a nonnegative integer multiple of the other, 
or if $\{\lambda_1,\ldots,\lambda_n\}=\{\lambda_1',\ldots,\lambda_n'\}$ as multisets.

Taking the level operator to be $d$, we get the balanced $\Z$-grading
\begin{align}
 L=\bigoplus_{i=-m}^mL_i,\qquad
 m=\max\{\lambda_1,\ldots,\lambda_n\},
\end{align}
where
\begin{align}
 L_0 =\C c\oplus\C d\oplus\spa\{b_j,b_j^\dagger\,|\,\lambda_j=0;\,j=1,\ldots,n\}
\end{align}
and (for $i=1,\ldots,m$)
\begin{align}
 L_{-i}=\spa\{b_j\,|\,\lambda_j=i;\,j=1,\ldots,n\},\qquad
 L_i =\spa\{b^\dagger_j\,|\,\lambda_j=i;\,j=1,\ldots,n\}.
\end{align}
It follows that
\begin{align}
 \g_0=\C d,\qquad
 \kf_0=\C c\oplus\spa\{b_j,b_j^\dagger\,|\,\lambda_j=0;\,j=1,\ldots,n\},
\end{align}
so
\begin{align}
 C_0(\g_0)=L_0,\qquad \Cf=\kf_0.
\end{align}
Due to the multiplicity free condition (iii) in Definition \ref{Def:normal},
$L$ is normal if and only if $\lambda$ is multiplicity free (that is, $|\{\lambda_1,\ldots,\lambda_n\}|=n$).
As $G$ defined in \eqref{G} is seen to be given by
\begin{align}
 G=\C c\oplus\C d\oplus\spa\{b_j,b_j^\dagger\,|\,\lambda_j\neq0;\,j=1,\ldots,n\},
\end{align}
it follows that $L$ does not satisfy the completeness condition (v) in Definition \ref{Def:regular} 
if $\lambda_j=0$ for at least one $j$. Thus, $L$ is non-regular in that case.

\section{Representation theory}
\label{Sec:Rep}

Throughout the remainder of the paper, we assume that $L$ is a normal $\Z$-graded Lie algebra, 
and that $d \in Z(\g_0)$, c.f.~Theorem~\ref{Theo:d}. 

Triangular decompositions play important roles in the description and representation theory of Lie 
algebras \cite{MP95}. In our case, $L$ admits the triangular decomposition
\begin{align}
 L = L_-\oplus L_0 \oplus L_+,
 \qquad
 L_-:=\bigoplus_{i<0}L_i,\qquad L_+:=\bigoplus_{i>0}L_i.
\end{align}
We note that 
\begin{align}
{\overline L}_{\pm}:= L_0 \oplus L_{\pm}
\end{align}
are Lie subalgebras of $L$, and we denote the universal enveloping algebras of 
$L,L_0,L_{\pm},{\overline L}_{\pm}$ by $U,U_0,U_{\pm},{\overline U}_{\pm}$, respectively. 
In view of the PBW theorem, we have the decompositions
\begin{align}
U = U_{-}U_0U_{+} = U_{+}U_0U_{-},\qquad
U = {\overline U}_{-}L_{-} \oplus U_0 \oplus UL_+,
\label{UUU}
\end{align}
where the last expression also serves as a (two-sided) $U_0$-module decomposition.

For any subset $W$ of the $L$-module $V$, the $L$-module generated by $W$ is denoted by
\begin{align}
 UW:=\spa\{uw\,|\,u \in U,\,w\in W\}.
\end{align}
In the case $W=\{v\}$, we may write $UW=Uv$.
\begin{Definition}
An $L$-module $V\!$ is 
(i) said to be locally $L_+$-finite if, for each $v\in V$, $U_+v$ is finite-dimensional, and
(ii) called a weight module if
\begin{align}
 V\cong\bigoplus_{\lambda\in H^*}V_\lambda,\qquad V_\lambda:=\{v\in V\,|\,hv=\lambda(h)v,\ \forall h\in H\}.
\end{align}
\end{Definition}
\noindent
Moreover, we let $D_{0}^{+} \subset H^*$ denote the set of $\g_0$-dominant weights and $V(\Lambda)$ 
an irreducible highest-weight $L$-module of highest weight $\Lambda\in H^*$. 
In fact, by Theorem~\ref{Th:V1V2} below, this module is unique.

\subsection{Category $\Zc$}
\label{Sec:CatZ}

We recall that $L$ is assumed normal (see Definition~\ref{Def:normal}) and that the level operator is an element
of $Z(\g_0)$. In fact, $d\in Z(L_0)\cap\g_0$.
\begin{Definition}
\label{Def:CatZ}
Category $\Zc$ is defined to be the full category of $L$-modules whose objects are the modules 
$V\!$ satisfying the following three conditions:
\begin{itemize}
\item[$(\Zc1)$]
$V\!$ is finitely generated.
\item[$(\Zc2)$]
$d$ acts diagonalisably on $V\!$.
\item[$(\Zc3)$]
The spectrum of $d$ is bounded from above.
\end{itemize}
\end{Definition}
\begin{Proposition}
Every $V\in\Zc$ is Noetherian and $L_+$-finite.
\end{Proposition}
\begin{proof}
The universal enveloping algebra of a finite-dimensional Lie algebra is Noetherian \cite{Dix96}, 
and since every finitely generated module over a Noetherian ring is Noetherian, $V$ is Noetherian.
Property $(\Zc3)$ in Definition~\ref{Def:CatZ} implies that $V$ is $L_+$-finite.
\end{proof}
\begin{Proposition}
Every module in $\Zc$ is isomorphic to a direct sum of modules admitting a $\Z$-gradation of the form
\begin{align}
 V=\bigoplus_{i \leq N}V_i,\qquad L_iV_j \subseteq V_{i+j},\qquad N\in\Z,
\label{VinZ}
\end{align}
where $V_n=\z$ for all $n>N$.
\end{Proposition}
\begin{proof}
By $(\Zc2)$ in Definition~\ref{Def:CatZ}, $V$ is a direct sum of $d$-eigenspaces. By $(\Zc1)$, 
every $d$-eigenvalue $\gamma$ will satisfy $\gamma\in\Z+\epsilon_\ell$ for one of finitely many 
possible scalars $\epsilon_\ell$. The result now follows from the boundedness property $(\Zc3)$.
\end{proof}
\begin{Proposition}
Let $V\in\Zc$. On any irreducible $L_0$-submodule of $V\!$, $d$ acts as a scalar multiple of the identity. 
\end{Proposition}
\begin{proof}
This follows from $(\Zc2)$ in Definition~\ref{Def:CatZ} and the fact that $d\in Z(L_0)$.
\end{proof}
\noindent
\textbf{Remark.}
For each nonzero $V\in\Zc$, we have $L_-V\subseteq V$, and if $V$ is of the form \eqref{VinZ}, then
\begin{align}
 V_N \cap L_{-}V = \z.
\label{LmV}
\end{align}
\noindent
\textbf{Remark.}
As a module over itself, we have $L\in\Zc$, and $L$ is of the form \eqref{VinZ} with $N=l$. Moreover, $U_+$ and $U_-$ 
admit $\Z$-gradations according to level:
\begin{align}
 U_+=\bigoplus_{i \geq 0}U_+^i,\qquad U_-=\bigoplus_{i \leq 0}U_-^i,
\label{Um}
\end{align}
where
\begin{align}
 U_\pm^i:= \{u \in U_\pm\,|\,[d,u] = iu\},\qquad i\in\Z,
 \end{align}
noting that $U_\pm^0=\C$.

\subsection{Primary component}

\begin{Definition}
For $V\in\Zc$, the primary component is defined as
\begin{align}
 \Vp:= \{v \in V\,|\,L_+v=0\}.
\end{align}
\end{Definition}
\noindent
\textbf{Remark.}
$\Vp$ is a nontrivial $L_0 $-submodule and, with reference to \eqref{VinZ}, $V_N\subseteq\Vp$.
\medskip

\noindent
Following previous work \cite{Gould91} on Lie superalgebras, we now introduce the following notion.
\begin{Definition}
An $L$-module $V\in\Zc$ is called standard if 
\begin{align}
 V=U\Vp.
\label{VUV0}
\end{align}
\end{Definition}
\noindent
\textbf{Remark.}
In view of \eqref{UUU}, a standard $L$-module $V$ may be written
\begin{align}
 V=U_{-}U_0U_{+}\Vp= U_{-}U_0\Vp= U_{-}\Vp= U_{-}L_{-}\Vp+\Vp.
\label{VUUUU}
\end{align}
We also note that any irreducible $V\in\Zc$ is standard.
\begin{Proposition}
For $V\in\Zc$, $U\Vp$ is the unique maximal standard submodule of $V$.
\end{Proposition}
\begin{proof}
Let $W \subseteq V$ be a standard submodule of $V$. Then, $W = U\Wp$, and since $\Wp\subseteq\Vp$,
it follows that $W\subseteq U\Vp$.
\end{proof}

For $V$ standard, with
\begin{align}
 V_0:=\Vp,
\end{align} 
we now define a sequence of $L_0 $-submodules $V_i$ recursively by 
\begin{align}
V_{-i}:=\begin{cases} L_{-i}V_0 + L_{-i+1}V_{-1} +\cdots+ L_{-1}V_{-i+1},\ &i=1,\ldots,k,
\\[.2cm]
 L_{-k}V_{k-i} + L_{-k+1}V_{k-i-1} +\cdots+ L_{-1}V_{-i+1},\ &i > k,
 \end{cases}
\end{align}
where $-k$ is the lower bound on $i$ in the $\Z$-grading of $L$ in Definition~\ref{Def:Zgraded}.
In the notation of \eqref{Um}, we have $V_i = U_{-}^{i}V_0$ for $i < 0$. Moreover,
\begin{align}
L_{i}V_{j} \subseteq V_{i+j},
\end{align}
where $V_{i+j}\equiv\z$ if $i+j>0$. 
\begin{Theorem}
Let $V$ be an irreducible standard $L$-module. Then, as an $L_0 $-module,
$V$ decomposes as
\begin{align}
 V=\bigoplus_{i \leq 0}V_i.
\label{VsumV}
\end{align}
\end{Theorem}
\begin{proof}
It suffices to prove that
\begin{align}
 V_j\cap \Big(\bigoplus_{i=j+1}^{0}V_{i}\Big) = \z,\qquad \forall j<0,
\label{VjV}
\end{align}
and we do that by induction on $-j$. For the induction start at $j=-1$, 
$V_{-1} \cap V_0 \subseteq U_{-}L_{-}V_0 \cap V_0$.
If $W:=V_0\cap U_-L_-V_0$ is nonzero, then, in view of the irreducibility of $V$,
\begin{align}
 V=UW=U_-L_-W+W\subseteq U_{-}L_-V_0 = L_{-}U_-V_0 \subseteq L_{-}V,
\end{align}
in contradiction to \eqref{LmV}. It follows that
\begin{align}
 V_0\cap U_-L_-V_0=\z,
\label{V0ULV}
\end{align}
so $V_{-1} \cap V_0=\z$.
For the induction step, suppose
\begin{align}
V_j \cap \Big(\bigoplus_{i=j+1}^{0}V_{i}\Big) = \z,\qquad \ell\leq j<0,
\end{align}
for some integer $\ell$, and let $v \in V_{\ell-1} \cap (\bigoplus_{i=\ell}^{0}V_{i})$. 
Then, by the induction hypothesis,
\begin{align}
L_nv \subseteq V_{\ell+n-1} \cap \Big(\bigoplus_{i=\ell+n}^{0}V_{i}\Big) = \z
\end{align}
for every $n>1$, so $L_+v=\z$. Hence, 
\begin{align}
v \in V_0 \cap V_{\ell-1} \subseteq V_0 \cap U_{-}L_{-}V_0 = \z.
\end{align}
\end{proof}
\begin{Theorem}
\label{Theo:Vstandard}
Let $V\!$ be a standard $L$-module. 
Then, $V\!$ is irreducible if and only if $\Vp$ is an irreducible $L_0 $-module.
\end{Theorem}
\begin{proof}
Suppose $\Vp$ is an irreducible $L_0 $-module and that $W \neq \z$ is an $L$-submodule of $V$ 
with primary component $\Wp$. Then, $\Wp\subseteq\Vp$ 
and hence by irreducibility of $\Vp$, $\Wp=\Vp$. Thus, $V = U\Wp\subseteq W$, hence $V = W$, 
so $V$ is irreducible. 

As to the converse, suppose $V$ is irreducible and let $\Wp\subseteq\Vp$ be a nonzero 
$L_0 $-submodule of $\Vp$. Since $V$ is irreducible, we then have the $L_0 $-module decomposition
\begin{align}
 V=U\Wp=\Wp\oplus U_{-}L_{-}\Wp,
\end{align}
where we have used the result
\begin{align}
 \Wp\cap U_{-}L_{-}\Wp \subseteq\Vp \cap U_{-}L_{-}\Vp = \z,
\end{align}
which follows from \eqref{V0ULV}. We now let $v_0 \in\Vp$ be arbitrary and write
\begin{align}
 v_0 = w_0 + v_1,\qquad w_0 \in\Wp,\qquad v_1 \in U_{-}L_{-}\Wp.
\end{align}
Then, $v_1 = v_0 - w_0 \in\Vp \cap U_{-}L_{-}\Vp = \z$, from which it follows that $v_0 = w_0 \in\Wp$. 
This shows that $\Vp \subseteq\Wp$ and hence $\Vp=\Wp$, so $\Vp$ is an irreducible $L_0 $-module.
\end{proof}
\begin{Corollary}
Let $V\in\Zc$. Then, $V$ is irreducible if and only if \,$V$\! is standard 
and $\Vp$ is an irreducible $L_0 $-module.
\end{Corollary}
\begin{proof}
Let $V\in\Zc$ be irreducible. Then, $V$ is standard and, by Theorem~\ref{Theo:Vstandard}, 
$\Vp$ is an irreducible $L_0 $-module.
The converse result follows immediately from Theorem~\ref{Theo:Vstandard}.
\end{proof}
\noindent
\textbf{Remark.}
For an irreducible standard $L$-module,
the $L_0$-modules $V_i$ occurring in the decomposition \eqref{VsumV} are given by
\begin{align}
V_i = \{v \in V\,|\, (d-\Delta)v = iv\},
\end{align}
where $\Delta$ is the eigenvalue of $d$ on $V_0=\Vp$. This makes the $\Z$-grading of $V$ explicit.  
In Proposition~\ref{Prop:VKM} below, we show that every irreducible $L$-module is uniquely characterised 
(up to isomorphism) by its primary component $\Vp$.

\subsection{Induced modules}
\label{Sec:Induced}

Let $W$ be an $L_0$-module; it becomes an ${\overline L}_{+}$-module by setting $L_+W= \z$. 
The corresponding induced $L$-module
\begin{align}
 K(W):= U_{-} {\otimes}_{{\overline U}_{+}}W
\label{KV0}
\end{align}
admits a $\Z$-gradation of the form \eqref{VinZ}:
\begin{align}
K(W) = \bigoplus_{i \leq 0}K_i,\qquad K_i:= U_{-}^{i} \otimes W,
\end{align}
where we note that
\begin{align}
U_{-}L_{-}K_0 = U_{-}L_{-} \otimes W= \bigoplus_{i \leq -1}K_i.
\end{align}

The following results summarise some important properties of the induced modules \eqref{KV0} 
in the case $W$ is irreducible as an $L_0$-module. 
With appropriate modifications, the proofs follow similar proofs in \cite{Gould91} 
(see Theorem 4.2 and Lemmas 4.1-4.3 therein).
\begin{Theorem}
\label{Theo:Wirr}
Let $W$ be an irreducible $L_0$-module. Then, the following holds:
\begin{itemize}
\item[{\rm (i)}]
$K(W)$ is indecomposable.
\item[{\rm (ii)}]
$K(W)$ contains a unique maximal $\Z$-graded submodule $M(W)$.
\item[{\rm (iii)}]
$M(W)$ is maximal in the set of all proper submodules of $K(W)$.
\item[{\rm (iv)}]
$K(W)/M(W)$ is an irreducible $L$-module with primary component isomorphic to $W$.
\end{itemize}
\end{Theorem}
\noindent
\textbf{Remark.} 
The submodules in Theorem~\ref{Theo:Wirr} (iii) are not assumed $\Z$-graded.
\begin{Proposition}
\label{Prop:VKM}
Let $V\in\Zc$ be irreducible. Then,
\begin{align}
V \cong K(\Vp)/M(\Vp).
\end{align}
\end{Proposition}
\begin{Proposition}
\label{Prop:V0V0}
Let $W$ and $W'$ be irreducible $L_0$-modules and $\phi_0:W\rightarrow W'$ 
an $L_0$-module isomorphism. 
Then, $\phi_0$ extends to an $L$-module isomorphism $\phi: K(W) \rightarrow K(W')$ which in turn 
induces an $L$-module isomorphism $K(W)/M(W)\to K(W')/M(W')$.
\end{Proposition}
\begin{Theorem}
Let $V,V'\in\Zc$ be irreducible. 
Then, $V\cong V'$ as $L$-modules if and only if $\Vp\cong \Vp'$ as $L_0$-modules.
\end{Theorem}
\begin{proof}
If $\Vp\cong\Vp'$ as $L_0$-modules, then
\begin{align}
 V \cong K(\Vp)/M(\Vp) \cong K(\Vp')/M(\Vp') \cong V'
\end{align}
as $L$-modules.
Conversely, if $\phi : V \rightarrow V'$ is an $L$-module isomorphism, 
then $\phi(\Vp)$ is a nonzero $L_0$-submodule of $V'$ satisfying
\begin{align}
L_{+}\phi(\Vp) = \phi(L_{+}\Vp) = \z.
\end{align}
It follows that $\phi(\Vp) \subseteq\Vp'$ and, since $\Vp'$ is irreducible, that $\phi(\Vp)=\Vp'$. 
Hence, $\phi$ restricts to an $L_0$-module isomorphism $\Vp\to\Vp'$.\end{proof}
\noindent
\textbf{Remark.}
Using the Killing decomposition
\begin{align}
 L_0 = \g_0 \oplus \kf_0,
\label{L0g0}
\end{align}
a $\g_0$-module $W$ readily becomes an $L_0$-module by setting
\begin{align}
 \kf_0W=\z.
\label{k0V0} 
\end{align}
This construction allows us to consider $L_0$-modules (and hence $L$-modules) induced from any
 $\g_0$-module. In the case $W$ is irreducible as a $\g_0$-module, we say that the ensuing $L_0$-module is 
{\it $\g_0$-irreducible}. Of particular interest is the case where $W$ is a finite-dimensional irreducible 
$\g_0$-module. However, with this construction, $Z(L)\subseteq\kf_0$ always acts trivially. A method which 
overcomes this drawback is discussed below.

\subsection{$L_0$-modules}
\label{Sec:L0}

Motivated by the important role played by $L_0$-modules above, we now consider the representation theory 
for $L_0$. We shall make extensive use of the Killing decomposition \eqref{L0g0} and the decompositions
\begin{align}\label{m0}
 L_0=\G_0\oplus\m_0,\qquad
 \G_0:=\g_0\oplus Z(L),\qquad
 \kf_0:= Z(L) \oplus \m_0.
\end{align}
Note that $\G_0$ is a reductive Lie algebra, while $\m_0$ is a $\g_0$-submodule of $\kf_0$.
For each $\alpha \in \Phi_0$, the corresponding root 
space satisfies $L_{\alpha} \subseteq \g_0 \oplus \m_0$. We denote the zero-weight space in $\m_0$ 
by $\zf_0$ and note that $C_0(H)=H\oplus\zf_0$,
so $C_0(H)\cap\kf_0=\pi_0$ where
\begin{align}
 \pi_0:=Z(L)\oplus\zf_0.
\end{align} 
As for semisimple Lie algebras, we have the following result.
\begin{Proposition}
\label{Theo:NH}
$N_0(H) = C_0(H)$.
\end{Proposition}
\begin{proof}
Observe that $C_0(H) \subseteq N_0(H)$. 
Hence, if $n \in N_0(H)$, the Killing decomposition gives $n = x + y$, where $x \in \g_0$ and $y\in \kf_0$. 
Then, $[y,H] = [n,H] - [x,H] \subseteq (H + \g_0)\cap K = K \cap(\g_0 \oplus Z(L))= Z(L)$. 
It follows that $y\in \kf_0$ has zero weight, so 
$y\in\pi_0$, hence $n \in \g_0 \oplus\pi_0$. But $n$ must also have zero weight,
so $n \in H \oplus\zf_0$ and hence $N_0(H) \subseteq  H \oplus\zf_0 = C_0(H)$. 
\end{proof}

\noindent
Since $C_0(H) \cap \kf_0$ is a (solvable) subalgebra of $L_0$, we also have
\begin{align}
[\zf_0,\zf_0] \subseteq\pi_0.
\end{align}
It follows that
\begin{align}
 \zf:= \zf_0+ [\zf_0,\zf_0]
\end{align}
is a solvable subalgebra containing $\zf_0$ (in fact, the smallest subalgebra containing $\zf_0$). 
With $Z_0:= Z(L) \cap \zf$, we have
\begin{align}
Z(L) = Z' \oplus Z_0, \qquad
\pi_0 = Z' \oplus \zf, \qquad
\zf = Z_0 \oplus \zf_0.
\end{align}

\begin{Proposition}
If $L$ is a regular $\Z$-graded Lie algebra, then $\zf$ is nilpotent.
\end{Proposition}
\begin{proof}
As $L$ is regular, $\kf_0$ is nilpotent, so the result follows from $\zf \subseteq \kf_0$.
\end{proof}
Since $\zf \subseteq \kf_0$, we have for any $i$ and $\beta \in \Phi_i$,
\begin{align}
[\zf,L_{\beta} \cap \g_i] \subseteq L_{\beta} \cap \kf_i.
\end{align}
The algebra $\zf$ is thus intimately connected with the existence of root multiplicities, as illustrated by
the following result.
\begin{Proposition}
If $L$ is regular and all roots in $\Phi_0$ appear with multiplicity one, then $\zf = \zf_0= \z$.
\end{Proposition}
\begin{proof}
By definition, $[H,\zf_0] = \z$. Suppose $L_{\alpha} \subseteq \g_0$ for some $\alpha\in\Phi_0$.
Then, $[\zf_0,L_{\alpha}] \subseteq L_{\alpha} \cap K \subseteq \g_0 \cap K = \z$, so
$\zf_0\subseteq C_0(\g_0) \cap K = Z(L)$, hence $\zf_0= \z$. If no $L_{\alpha} \subseteq \g_0$ exists, 
then $\g_0=\h_0$, $\G_0=H$, hence $[\zf_0, \g_0] = [\zf_0,\h_0] = \z$, so
$\zf_0\subseteq C_0(\g_0) \cap K = Z(L)$ and hence $\zf_0= \z$.
\end{proof}

\noindent
Now set $H':= \h_0 \oplus Z'$, which is a subalgebra of $H = H' \oplus Z_0$. 
Then, we have the triangular decomposition
\begin{align}
L_0 = \g_{-} \oplus H' \oplus \g_+,
\end{align}
where 
\begin{align}
\g_{-}:= \bigoplus_{\alpha < 0}^{\Phi_0}L_{\alpha},\qquad  
\g_{+}:= \bigoplus_{\alpha > 0}^{\Phi_0}L_{\alpha} \oplus \zf.
\end{align}
We note that $\g_-$ and $\g_+$ give rise to (nilpotent respectively solvable) Lie subalgebras of $L_0$.

We use the convention that a weight vector is nonzero. 
\begin{Definition}
A maximal-weight vector in an $L$-module $W$ is a weight vector $v_{+}$ such that $\g_+ v_{+} \subseteq \C v_+$.
\end{Definition}

\noindent
\textbf{Remark.}
For $0 < \alpha \in \Phi_0$, we have $L_{\alpha} v_+ = \z$. 
Also, $v_+$ determines a one dimensional representation $\chi$ of $\zf$, which in turn extends to an algebra homomorphism 
$\chi : U(\zf) \longrightarrow \C$. We call $\chi$ the \textit{characteristic} of $v_+$. 
If $\Lambda$ is the weight of the maximal-weight vector, then $\Lambda(h) = \chi(h)$ for all $h \in Z_0$.
\medskip

\noindent
\textbf{Remark.}  
In the case of the Schr{\"o}dinger and conformal Galilei algebras discussed in Section~\ref{Sec:Ex}, 
$\zf$, if nonzero, is one-dimensional and hence trivially gives rise to an abelian Lie algebra $\zf =\zf_0$. 
In that case, $Z_0 = \z$ and $H = H'$. 

\begin{Definition}
An $L_0$-module $W$ is called standard cyclic if $W= U(L_0)v_{+}$ for some
maximal-weight vector $v_+\in W$.
\end{Definition}

\noindent
Obviously, any irreducible module with a maximal-weight vector is standard cyclic. 
One of our aims is to show that such an irreducible module is characterised, up to isomorphism, by its highest weight and characteristic. 
The following result follows from standard arguments. It uses the usual weight ordering induced by the 
positive roots, where, for $\mu,\lambda\in H^*$, $\mu\leq\lambda$ if $\lambda-\mu$ is in the positive root lattice.
\begin{Proposition}
\label{Prop:cyclic}
Let $W= U(L_0)v_{+}$ be standard cyclic with highest-weight vector $v_{+}$ of weight $\Lambda \in H^{*}$.
Then, the following holds:
\begin{itemize}
\item[{\rm (i)}]
The weight $\Lambda$ occurs with unit multiplicity in $W$.
\item[{\rm (ii)}]
Every weight $\mu$ in $W$ satisfies $\mu\leq\Lambda$.
\item[{\rm (iii)}]
$W$ is indecomposable.
\end{itemize}
\end{Proposition}
\begin{Theorem}
\label{Th:V1V2}
Let $W_1$ and $W_2$ be irreducible highest-weight $L_0$-modules with the same highest weight and characteristic.
Then, $W_1 \cong W_2$.
\end{Theorem}
\begin{proof}
Let $\Lambda \in H^{*}$ be the highest weight, and let
$w_1 \in W_1$ and $w_2 \in W_2$ be maximal-weight vectors of weight $\Lambda$ and characteristic $\chi$.
For $i = 1,2$, let $I(w_i) = \{u \in U(L_0)\,|\,uw_i = 0\}$ be the annihilator ideal of $w_i$ in $U(L_0)$.
To show $I(w_2) \subseteq I(w_1)$, assume it does not hold. Then, there exists $u \in U(L_0)$ such that 
$uw_2 =0$ and $uw_1 \neq 0$. 
Let $C = U(L_0)w_0$ where $w_0=w_1+w_2$.
Then,
\begin{align}
0 \neq uw_0 = uw_1 + uw_2 = uw_1 \in W_1 \cap C,
\end{align}
so irreducibility implies $W_1 \subseteq C$. Thus, $w_1 \in C$, so $w_2 = w_0 - w_1 \in C$. 
This implies $W_1 \oplus W_2 \subseteq C$, hence $C = W_1 \oplus W_2$, in contradiction to 
Proposition~\ref{Prop:cyclic}. 
Thus, $I(w_2) \subseteq I(w_1)$. Similarly, $I(w_1) \subseteq I(w_2)$, hence $I(w_2) = I(w_1)$. 
We may thus define an $L_0$-module homomorphism 
$\psi : W_1 \rightarrow W_2$, $\psi(uw_1) = uw_2$ for all $u \in U(L_0)$. 
As $\psi$ is seen to be bijective, this completes the proof.
\end{proof}
\noindent

\begin{Definition}
The irreducible highest-weight $L_0$-module
with highest weight $\Lambda$ and characteristic $\chi$ is denoted by $V_0(\Lambda,\chi)$.
\end{Definition}
\noindent
\textbf{Remark.}
Any highest-weight vector $v_{+}$ of weight $\Lambda \in H^{*}$ determines a one-dimensional 
representation of the $L_0$-subalgebra $\g_{+}$. We thus have the following induced modules.
\begin{Definition}
Let $W$ be an $L_0$-module, and suppose $v_+\in W$ 
is a highest-weight vector of weight $\Lambda\in H^*$ and charcterisic $\chi$.
Then, the corresponding Verma $L_0$-module is defined as
\begin{align}
 M_0(\Lambda, \chi):= U(\g_{-}) \otimes_{\g_+} \C v_+.
\label{M0L}
\end{align}
\end{Definition}
\noindent
\textbf{Remark.}
The weight spectrum of $M_0(\Lambda, \chi)$ depends only on the highest weight $\Lambda$.
\medskip

\noindent
The $L_0$-module \eqref{M0L} belongs to the following category of importance in Section~\ref{Sec:CatO} below.
\begin{Definition}
Category $\Lc$ is defined to be the full category of $L_0$-modules whose objects are the modules 
$V\!$ satisfying the following three conditions:

\begin{itemize}
\item[$(\Lc1)$]
$V\!$ is finitely generated.
\item[$(\Lc2)$]
$V\!$ is a direct sum of finite-dimensional weight spaces.
\item[$(\Lc3)$]
$V\!$ is locally $\g_+$-finite.
\end{itemize}
\end{Definition}

\section{Characters}
\label{Sec:Chacters}

We recall that $L$ is assumed a normal $\Z$-graded Lie algebra.

\subsection{Category $\Obar$}
\label{Sec:Obar}

As discussed in Section~\ref{Sec:Characters2} below, modules in the following category admit descriptions in terms 
of characters.
\begin{Definition}
\label{Def:Obar}
Category $\Obar$ is defined to be the full category of $L$-modules whose objects are the modules 
$V\!$ satisfying the following three conditions:
\begin{itemize}
\item[$(\Obar1)$]
$V\!$ is finitely generated.
\item[$(\Obar2)$]
$V\!$ is a direct sum of finite-dimensional generalised weight spaces.
\item[$(\Obar3)$]
$V\!$ is locally $L_+$-finite.
\end{itemize}
\end{Definition}
\begin{Theorem}
Category $\Obar$ is a subcategory of Category $\Zc$.
\end{Theorem}
\begin{proof}
Since $d \in H$, 
a module in Category $\Obar$ is a direct sum of $\Z$-graded modules and thus belongs to Category $\Zc$. 
\end{proof}

\noindent
\textbf{Remark.}
All finite-dimensional modules from Category $\Zc$ belong to Category $\Obar$.

\subsection{Harish-Chandra modules}
\label{Sec:HC}

We say that a weight vector $v_{+}$ in an $L$-module $V$ is \textit{maximal} if
\begin{align}
(L_+ \oplus \g_{+})v_{+} \subseteq \C v_+.
\end{align}
Such a vector is an $L_0$-maximal weight vector such that $L_+ v_+ = \z$. 
Noting that a maximal weight vector $v_+$ of $V(\Lambda)$ is unique up to scaling, 
we let $\vf(\Lambda):= U(\g_0) v_+$ denote the (standard cyclic) $\g_0$-module generated by $v_+$.
Here, $U(\g_0)$ is the universal enveloping algebra of $\g_0$.
This $\g_0$-module is indecomposable \cite{Hum94},
and if $\vf(\Lambda)$ is finite-dimensional, then it is irreducible. 
We refer to $\vf(\Lambda)$ as the \textit{maximal $\g_0$-component} of $V(\Lambda)$, and
$V(\Lambda)$ is said to be \textit{strongly irreducible} if $\vf(\Lambda)$ is an irreducible $\g_0$-module.

Recall that $\G_0:=\g_0\oplus Z(L)$ is a reductive Lie subalgebra of $L_0$ 
which contains the Cartan subalgebra $H$, and note that if $V(\Lambda)$ is an irreducible 
$L$-module with highest weight $\Lambda \in H^\ast$, then its maximal $\g_0$-component $\vf(\Lambda)$
is a $\G_0$-module.
\begin{Definition}
A module $V\in\Zc$ is called a Harish-Chandra module if it is a direct sum of 
finite-dimensional irreducible $\G_0$-modules, each occurring with finite multiplicity. 
\end{Definition}
\noindent
\textbf{Remark.}
It follows from properties (i), (ii) and (v) in Definition~\ref{Def:regular} 
that $L$ as a module over itself is a Harish-Chandra module.
\begin{Theorem}
Let $V\in\Zc$ be  irreducible and contain a finite-dimensional irreducible $\G_0$-module.
Then, $V\!$ is a Harish-Chandra module.
\end{Theorem}
\begin{proof}
Using the PBW theorem, it can be shown 
\cite{Dix96} that $V$ is a direct sum of finite-dimensional irreducible $\G_0$-modules 
which necessarily occur with finite multiplicity in $V$.
\end{proof}
\noindent
\textbf{Remark.}
An irreducible $V\in\Zc$ either contains no irreducible finite-dimensional $\G_0$-submodules or
is a direct sum of irreducible finite-dimensional $\G_0$-modules.
\begin{Theorem}
\label{Theo:irredHC}
Let $V\in\Zc$ be locally $\g_+$-finite. Then, $V\!$ is an irreducible Harish-Chandra module if and only if
$V$ is isomorphic to a strongly irreducible $L$-module $V(\Lambda)$ for some $\Lambda \in D_{0}^{+}$.
\end{Theorem}
\begin{proof}
 $V$ admits a $\Z$-graded decomposition of the form $V = \bigoplus_{i \leq 0}V_i$ with 
$L_iV_j \subseteq V_{i+j}$. Each $V_i$ constitutes a direct sum of finite-dimensional 
irreducible $\G_0$-modules. If $V$ is irreducible, we may assume its primary component 
$\Vp$ is an irreducible $L_0$-module whose weights are bounded from above. 
It follows that $V$ admits a maximal weight vector, say of highest-weight $\Lambda$, 
which generates the standard cyclic $\G_0$-module $\vf(\Lambda)$. 
Since $V$ is completely reducible as a $\G_0$-module, $\vf(\Lambda)$ 
must be an irreducible $\G_0$-module which is finite-dimensional, so $\Lambda \in D_{0}^{+}$. 
It follows that $V$ is a strongly irreducible $L$-module with highest weight $\Lambda$. 
Conversely, if $V\cong V(\Lambda)$ is strongly irreducible and $\Lambda \in D_{0}^{+}$, 
then the maximal $\g_0$-component of $V$ must be a finite-dimensional irreducible $\g_0$-module
that is irreducible as a $\G_0$-module. 
Hence, $V$ is an irreducible Harish-Chandra module.
\end{proof}
\begin{Theorem}
If $V\in\Zc$ is a Harish-Chandra module, then $V\in\Obar$.
\end{Theorem}
\begin{proof}
As to $(\Obar1)$, $V$ is finitely generated by $(\Zc1)$. 
As to $(\Obar2)$, since $V$ is a direct sum of irreducible finite-dimensional $\G_0$-modules, 
$V$ is a direct sum of weight spaces, and by $(\Zc1)$, these weight spaces are finite-dimensional.
As to $(\Obar3)$, it is a consequence of $(\Zc3)$ and $(\Zc3)$ that $U(L_+)v$ is finite-dimensional for any $v \in V$.
\end{proof}

\subsection{Category $\Oc$}
\label{Sec:CatO}

\begin{Definition}
Category $\Oc$ is defined to be the full category of $L$-modules whose objects are the modules 
$V\!$ satisfying the following three conditions:
\begin{itemize}
\item[$(\Oc1)$]
$V\!$ is finitely generated.
\item[$(\Oc2)$]
$V\!$ is a weight module.
\item[$(\Oc3)$]
$V\!$ is locally $L_+ \oplus \g_+$-finite.
\end{itemize}
\end{Definition}
\noindent
We readily have the following result.
\begin{Proposition}
\label{WKW}
If $W \in\Lc$, then $K(W)\in\Oc$.
\end{Proposition}
\noindent
\textbf{Remark.}
A module in Category $\Oc$ is a direct sum of weight spaces with weight spectrum bounded from above.
This upper bound is relative to the usual partial ordering induced by the positive roots. 
Clearly, Category $\Oc$ is a subcategory of Category $\Obar$.
In particular, we may consider the irreducible $L$-module $V(\Lambda)$ with highest weight 
$\Lambda\in H^*$ which admits a $\g_0$ maximal component $\vf(\Lambda)$ which is 
a standard cyclic $\g_0$-module of the same highest weight. 
\begin{Definition}
Let $M(\Lambda)$ be a Verma $\g_0$-module of highest weight $\Lambda \in H^*$, extended to an $L_0$-module
by setting $\kf_0 M(\Lambda)=\z$. Then, the corresponding Verma $L$-module is defined as the induced module 
\begin{align}
 K(\Lambda):= U_{-} {\otimes}_{{\overline U}_{+}} M(\Lambda)
\label{KL}
\end{align}
\end{Definition}
\begin{Proposition}
\label{Prop:VLO}
Every Verma $L$-module is in Category $\Oc$.
\end{Proposition}
\begin{proof}
Using Proposition~\ref{WKW}, this follows from the fact that the $L_0$-module $M(\Lambda)$ belongs to category $\Lc$.
\end{proof}
\noindent
\textbf{Remark.}
The situation in Proposition~\ref{Prop:VLO} is quite different to the case the maximal $\Z$-graded component 
$V_0 = M_0(\Lambda)$ is a Verma $L_0$-module. 
$K(\Lambda)$ is also distinct from the $L$-module $K(V_0(\Lambda))$ induced from 
the $\g_0$-irreducible $L_0$-module $V_0(\Lambda)$.
\medskip

\noindent
\textbf{Remark.}
All finite-dimensional Harish-Chandra $L$-modules belong to Category $\Oc$.
As for simple Lie algebras, Category $\Oc$ is closed under tensor products with finite-dimensional $L$-modules from Category $\Oc$. 
\begin{Proposition}
Let $V$ be a finitely generated Harish-Chandra module on which $L_+ \oplus \g_+$ is locally nilpotent. Then, $V\in\Oc$. 
\end{Proposition}
\begin{proof}
$V$ is generated by a finite number of weight vectors. Each of these will generate an $L$-module which is 
a direct sum of finite-dimensional weight spaces with weights bounded from above. 
It follows that $V \in \Oc$.
\end{proof}

\subsection{Characters}
\label{Sec:Characters2}

\begin{Definition}
Let $V\in\Obar$. Its character is defined as
\begin{align}
ch_{q}^{V}:= \sum_{\nu}q^{\nu},
\end{align}
where the sum is over all weights $\nu$ (including multiplicities) in $V$.
\end{Definition}
\noindent
\textbf{Remark.}
The finiteness condition $(\Obar2)$ in Definition~\ref{Def:Obar} ensures that $ch_{q}^{V}$ is well-defined.
\medskip

\noindent
To compute the character of a module in $\Obar$, we first note that each root sum of the form
\begin{align}
\rho_i:= -\tfrac{1}{2}\sum_{\beta \in \Phi_{-i}}\beta,\qquad i > 0,
\end{align}
satisfies
\begin{align}
(\rho_i,\alpha) = 0,\qquad \forall \alpha \in \Phi_0,
\end{align}
and is therefore fixed by the Weyl group $\Wc_0$. Using these root sums, we introduce the {\em Weyl vector}
\begin{align}
 \rho:=\sum_{i=0}^k\rho_i= \rho_0 -\tfrac{1}{2}\sum_{\beta \in \Phi_1^-}\beta,
\label{rho}
\end{align}
where $\rho_0$ denotes the half-sum of the positive roots of $\g_0$.
We also find it convenient to define
\begin{align}
\hat{\Phi}^- :=\Phi_0^{s,-} \cup \Phi_1^{-}
\end{align}
and ($\epsilon\in\{0,1\}$)
\begin{align}
 P_\epsilon^-:=\prod_{\alpha\in \Phi_\epsilon^-}\frac{-1}{q^{\frac{\alpha}{2}}-q^{-\frac{\alpha}{2}}},\qquad
 P_0^{s,\pm}:= \prod_{\alpha \in \Phi_0^{s,\pm}}\frac{\pm1}{q^{\frac{\alpha}{2}}-q^{-\frac{\alpha}{2}}},\qquad
 \hat{P}^-:=\prod_{\alpha\in \hat{\Phi}^-}\frac{-1}{q^{\frac{\alpha}{2}}-q^{-\frac{\alpha}{2}}}.
\label{P}
\end{align}

We now let $W$ be an $L_0$-module. The character of the corresponding induced $L$-module $K(W)$ is 
thus given by 
\begin{align}
 ch_{q}^{K(W)}=\sum_{\beta\in\Phi_1^-}\sum_{n_{\beta}\in\Nb_0}q^{\sum_{\beta}n_{\beta}\beta}ch_{q}^{W},
\end{align}
where $ch_{q}^{W}$ is an $L_0$-character. It follows that
\begin{align}
ch_{q}^{K(W)}=\Big(\prod_{\beta \in \Phi_1^-}\frac{1}{1-q^\beta}\Big)ch_{q}^{W}
 =q^{\rho -\rho_0}P_1^-ch_{q}^{W}.
\label{chqKV0}
\end{align}
Two cases are of particular interest.

First, recall that the character of the Verma $\g_0$-module $W= M(\Lambda)$ is given by
\begin{align}
ch_{q}^{M(\Lambda)}=P_0^{s,+}q^{\Lambda + \rho_0}.
\end{align}
It follows that the character of the corresponding Verma $L$-module $K(\Lambda)$ is given by
\begin{align}
ch_{q}^{K(\Lambda)}=\hat{P}^-q^{\Lambda+\rho}.
\end{align}
This is contrasted to the case of a Verma  $L_0$-module $W = M_0(\Lambda)$,
whose character is given by
\begin{align}
ch_{q}^{W} =P_0^-q^{\Lambda+\rho_0'},
\qquad
\rho_0':= -\tfrac{1}{2}\sum_{\alpha < 0}^{\Phi_0}\alpha.
\end{align}

In the other case, let $W= W(\Lambda)$ be a finite-dimensional $\g_0$-irreducible $L_0$-module with 
highest weight $\Lambda \in D_{0}^+$, so that
\begin{align}
ch_{q}^{W(\Lambda)} =P_0^{s,+}\sum_{\sigma \in\Wc_0}sn(\sigma)q^{\sigma(\Lambda + \rho_0)}.
\end{align}
The induced module is now a Harish-Chandra module with character given by
\begin{align}
ch_{q}^{K(W(\Lambda))} =\hat{P}^-\sum_{\sigma \in\Wc_0}sn(\sigma)q^{\sigma(\Lambda + \rho)}.
\end{align}

\subsection{Quantum dimensions}

We define the $q$-dimension of an $L$-module $V$ in the usual way as
\begin{align}
\dim_q[V]:= tr_V(q^{h_\rho}),
\end{align}
and introduce the level generating function of an $L$-module $V$ as
\begin{align}
d_q[V]:= tr_V(q^d).
\label{d}
\end{align}
For $\Lambda\in H^\ast$, we find it convenient to introduce
\begin{align}
 P_\epsilon^-(\Lambda)
 :=\prod_{\alpha\in \Phi_\epsilon^-}\frac{1}{q^{-\frac{(\alpha,\Lambda)}{2}}-q^{\frac{(\alpha,\Lambda)}{2}}},
 \qquad
 \Lambda\notin\bigcup_{\alpha\in\Phi}\Pc_\alpha,
\end{align}
where $\Pc_\alpha$ is defined in \eqref{hyperplane}.
Similar evaluations based on the other products in \eqref{P}
yield expressions naturally denoted by $P_0^{s,\pm}(\Lambda)$ and $\hat{P}^-(\Lambda)$.

For an $L_0$-module $W$, we have
\begin{align}
\dim_q[K(W)]=q^{(\rho - \rho_0,\delta)} P_1^-(\rho)\dim_q[W],
\end{align}
where $\dim_q[W] = tr_W(q^{h_\rho})$ is the usual $q$-dimension of $W$. 
In the case $W= M(\Lambda)$ is a Verma $\g_0$-module, we thereby obtain 
\begin{align}
\dim_q[K(\Lambda)]
 =\hat{P}^-(\rho)\, q^{(\Lambda + \rho,\rho)}.
\end{align}
Similarly, in the case $W=W(\Lambda)$ is a finite-dimensional $\g_0$-irreducible $L_0$-module 
with highest weight $\Lambda \in D_{0}^+$, the $q$-dimension for the 
corresponding induced module is given by
\begin{align}
\dim_q[K(W(\Lambda))]
 =P_1^-(\rho) \sum_{\sigma \in\Wc_0}sn(\sigma)q^{(\sigma(\Lambda + \rho),\rho)}.
\end{align}
Utilising the denominator formula
\begin{align}
\sum_{\sigma \in\Wc_0}sn(\sigma)q^{\sigma(\rho)} = P_0^{s,+},
\end{align}
we thereby obtain 
\begin{align}
\dim_q[K(W(\Lambda))]
 =P_1^-(\rho) \dim_{q}^0[W(\Lambda)],
\end{align}
where
\begin{align}
 \dim_{q}^0[W(\Lambda)]
 =\prod_{\alpha \in \Phi_0^{s,-}}\frac{q^{\frac{1}{2}(\Lambda+\rho,\alpha)}
  -q^{-\frac{1}{2}(\Lambda+\rho,\alpha)}}{q^{\frac{1}{2}(\rho,\alpha)}-q^{-\frac{1}{2}(\rho,\alpha)}}
\end{align}
is the usual $q$-dimension for the irreducible $\g_0$-module $W(\Lambda)$. We thus arrive at the interesting formula
\begin{align}
 \dim_q[K(W(\Lambda))] 
 =\frac{\hat{P}^-(\rho)}{P_0^{s,-}(\Lambda+\rho)}
 =\frac{\prod_{\alpha \in \Phi_0^{s,-}}(q^{-\frac{1}{2}(\Lambda+\rho,\alpha)}
 -q^{\frac{1}{2}(\Lambda+\rho,\alpha)})}{\prod_{\beta \in \hat{\Phi}^-}(q^{-\frac{1}{2}(\rho,\beta)}
    -q^{\frac{1}{2}(\rho,\beta)})}.
\end{align}

For the alternative $q$-dimension \eqref{d}, we see that
\begin{align}
d_q[K(W)] = q^{(\rho - \rho_0,\delta)}P_1^-(\delta)\,tr_W(q^d).
\end{align}
We now suppose that $d$ takes the constant value $\xi$ on $W$, in which case
\begin{align}
tr_W(q^d)= q^{\xi}\dim W.
\end{align}
This is only well-defined if $W$ is finite-dimensional, in which case 
\begin{align}
d_q[K(W)] =q^{(\rho,\delta)+\xi}P_1^-(\delta) \dim W.
\end{align}
Since $\Phi_1^- = \Phi_{-k} \cup\cdots\cup \Phi_{-1}$ and for $\beta \in \Phi_{-i}$, we have $(\beta,\delta) = -i$, 
it follows that
\begin{align}
 \big(P_1^-(\delta)\big)^{-1}=q^{(\rho,\delta)}\prod_{\beta \in \Phi_{1}^-}(1-q^{(\delta,\beta)})
= q^{(\rho,\delta)}\prod_{i=1}^{k}\prod_{\beta \in \Phi_{-i}}(1 - q^{-i})
= q^{(\rho,\delta)}\prod_{i=1}^{k}(1 - q^{-i})^{\dim L_{-i}}.
\end{align}
We thus obtain the interesting expression
\begin{align}
d_q[K(W)] = q^{\xi}\dim W\prod_{i=1}^{k}\Big(\frac{1}{1 - q^{-i}}\Big)^{\!\dim L_{-i}},
\end{align}
which can be readily made fully explicit using Weyl's dimension formula for $W$.

\section{Outlook}
\label{Sec:Concl}

Based on the Killing decomposition,
we have presented new advances in the theory of finite-dimensional $\Z$-graded Lie algebras. 
The algebra structure has been examined, and the associated representation theory has been 
extensively developed, including the introduction of new categories of $\Z$-graded modules. 
Special attention has been paid to a class of highest-weight induced modules whose characters have been 
determined using adaptations of traditional techniques. 

It is of great interest to consider applications of these advances to further investigate the representation theory 
of $\Z$-graded Lie algebras of physical relevance, with especially irreducible $\Z$-graded Lie algebras
appearing in a wide variety of contexts. While there has been significant development of the representation theory 
in certain special cases, notably the Schr\"odinger and conformal Galilei algebras for which there exists a large body of literature, 
their $\Z$-graded structures have not been exploited. It is of particular interest in this regard to examine the associated root systems 
and the various characters as $q$-series, including any connections with special functions. 

In light of Theorem~\ref{Theo:irredHC}, we intend to explore the presence of irreducible Harish-Chandra modules 
in association with unitary representations corresponding to a compact real form of $\g_0$.
We also plan to consider $\Z$-gradations of \textit{Lie superalgebras}, 
including the physically important superconformal Galilei algebras.

\addcontentsline{toc}{section}{Acknowledgements}
\subsection*{Acknowledgements}

This work was supported by the Australian Research Council under the Discovery Project scheme, 
project numbers DP160101376 and DP200102316.


\end{document}